\documentclass[12pt]{amsart}
\usepackage{amsmath,amssymb,amsthm}
\usepackage{color}
\usepackage{hyperref}

\newtheorem{lem}{Lemma}[section]
\newtheorem{prop}{Proposition}[section]

\newtheorem{thm}{Theorem}[section]

\theoremstyle{definition}
\newtheorem{definition}{Definition}[section]

\theoremstyle{remark}

\theoremstyle{remark}
\newtheorem{remark}{Remark}[section]
\numberwithin{equation}{section}

\newcommand{\C}{{\mathbb C}}
\newcommand{\N}{{\mathbb N}}

\newcommand{\R}{{\mathbb R}}

\definecolor{blu}{rgb}{0,0,1}

\title[Existence and instability of standing waves]{Existence and instability of standing waves  with prescribed norm for a class of Schr\"odinger-Poisson  equations}

\author[Jacopo Bellazzini]{Jacopo Bellazzini}%$^\dagger$}
\address{Jacopo Bellazzini
\newline\indent
Universit\` a di Sassari
\newline\indent
Via Piandanna 4, 07100 Sassari, Italy}
\email{jbellazzini@uniss.it}

\author[Louis Jeanjean]{Louis Jeanjean}%$^*$}
\address{Louis Jeanjean
\newline\indent
Laboratoire de Math\'ematiques (UMR 6623)
\newline\indent
Universit\'{e} de Franche-Comt\'{e}
\newline\indent
16, Route de Gray 25030 Besan\c{c}on Cedex, France}
\email{louis.jeanjean@univ-fcomte.fr}

\author[Tingjian Luo]{Tingjian Luo}%$^*$}
\address{Tingjian Luo
\newline\indent
Laboratoire de Math\'ematiques (UMR 6623)
\newline\indent
Universit\'{e} de Franche-Comt\'{e}
\newline\indent
16, Route de Gray 25030 Besan\c{c}on Cedex, France}
\email{tingjian.luo@univ-fcomte.fr}

\begin{document}
\subjclass[2000]{35J50, 35Q41, 35Q55, 37K45}

\keywords{Schr\"odinger-Poisson equations, standing waves, orbital instability, variational methods}

\begin{abstract}
In this paper we study the existence and the instability of standing waves with prescribed $L^2$-norm for a class of Schr\"odinger-Poisson-Slater equations in $\R^{3}$
%orbitally stable standing waves with arbitray  charge for  the following  Schr\"odinger-Poisson type equation
\begin{equation}\label{evolution1}
i\psi_{t}+ \Delta \psi - (|x|^{-1}*|\psi|^{2}) \psi+|\psi|^{p-2}\psi=0 %\ \  \text{ in } \R^{3},
\end{equation}
when  $p \in (\frac{10}{3},6)$. To obtain such solutions we look to critical points of the energy functional
$$F(u)=\frac{1}{2}\left \| \triangledown u \right \|_{L^{2}(\mathbb{R}^3)}^2+\frac{1}{4}\int_{\mathbb{R}^3}\int_{\mathbb{R}^3}\frac{\left | u(x) \right |^2\left | u(y) \right |^2}{\left | x-y \right |}dxdy-\frac{1}{p}\int_{\mathbb{R}^3}\left | u \right |^pdx $$
on  the constraints given by
$$S(c)= \{u \in H^1(\mathbb{R}^3) :\   \left \| u \right \|_{L^2(\R^3)}^2=c, c>0 \}.$$
For the values $p \in (\frac{10}{3}, 6)$ considered, the functional $F$ is unbounded from below on $S(c)$ and the existence of critical points is obtained by a mountain pass argument developed on $S(c)$. We show that critical points exist provided that $c>0$ is sufficiently small and that  when $c>0$ is not small a non-existence result is expected. Concerning the dynamics we show for initial condition $u_0\in H^1(\R^3)$ of the associated Cauchy problem with $\|u_0\|_{2}^2=c$ that the mountain pass energy level $\gamma(c)$ gives a threshold for global existence. Also the strong instability of standing waves at  the mountain pass energy level is proved. Finally we  draw a comparison between the Schr\"odinger-Poisson-Slater equation and the classical nonlinear Schr\"odinger equation.
\end{abstract}
\maketitle
\section{Introduction}
In this paper we prove the existence and the strong instability of standing
waves for the following Schr\"odinger-Poisson-Slater equations:
\begin{equation}\label{main}
 i\partial_t u+\Delta u  - (\left | x \right |^{-1}\ast \left | u \right |^2) u+|u|^{p-2}u=0 \  \mbox{ in } \ \mathbb{R} \times \mathbb{R}^{3}.
\end{equation}
This class of Schr\"odinger type equations with a repulsive nonlocal Coulombic potential is obtained by approximation of the  Hartree-Fock equation describing a quantum mechanical system of many particles, see for instance \cite{BA}, \cite{LS}, \cite{L2}, \cite{MA}.
We look for standing waves solutions of \eqref{main}. Namely for solutions in the form 
$$u(t,x)= e^{-i\lambda t}v(x),$$ where $\lambda \in \R$. Then the function $v(x)$ satisfies the equation
\begin{equation}\label{eq}
-\Delta v - \lambda v + (\left | x \right |^{-1}\ast \left | v \right |^2) v-|v|^{p-2}v=0 \ \mbox{ in } \  \mathbb{R}^{3}.
\end{equation}
The case where $\lambda \in \R$ is a fixed and assigned parameter has been extensively studied in these last years, see e.g. \cite{AP}, \cite{DM}, \cite{HK2}, \cite{HK}, \cite{R} and the references therein. In this case critical points of the functional
defined in $H^{1}(\mathbb R^{3})$
%\begin{equation}\label{libero}
$$
J(u): =\frac{1}{2}\int_{\R^{3}}|\nabla u|^{2}dx-\frac{\lambda}{2}\int_{\R^{3}}|u|^{2}dx+\frac{1}{4}\int_{\mathbb{R}^3}\int_{\mathbb{R}^3}\frac{\left | u(x) \right |^2\left | u(y) \right |^2}{\left | x-y \right |}dxdy-\frac{1}{p}\int_{\R^{3}}|u|^{p}dx$$
%\end{equation}
give rise to solutions  of \eqref{eq}. In the present paper, motivated by the fact that physics are often interested in ``normalized" solutions, we search for solutions with prescribed $L^2$-norm. A solution of \eqref{eq} with $\left \| u \right \|_{L^2(\R^3)}^2=c$ can be obtained as a constrained critical point of the functional
$$F(u):=\frac{1}{2}\left \| \triangledown u \right \|_{L^{2}(\mathbb{R}^3)}^2+\frac{1}{4}\int_{\mathbb{R}^3}\int_{\mathbb{R}^3}\frac{\left | u(x) \right |^2\left | u(y) \right |^2}{\left | x-y \right |}dxdy-\frac{1}{p}\int_{\mathbb{R}^3}\left | u \right |^pdx $$
\noindent
on the constraint $$S(c):= \{u \in H^1(\mathbb{R}^3) :\   \left \| u \right \|_{L^2(\R^3)}^2=c \}.$$
Note that in this case the frequency can not longer by imposed but instead appears as a Lagrange parameter. As we know, $F(u)$ is a well defined and $C^1$ functional on $S(c)$ for any
$p \in (2,6]$ (see \cite{R} for example). For $p\in(2,\frac{10}{3})$ the functional $F(u)$ is bounded from below and coercive on $S(c)$. The existence of minimizers for $F(u)$ constrained has been studied in the \cite{BS1}, \cite{BS}, \cite{SS}. It has been proved in \cite{SS}, using techniques introduced in  \cite{CattoL}, that minimizer exist for $p = \frac{8}{3}$ provided that $c \in (0,c_0)$ for a suitable $c_0>0$.
In \cite{BS} it is proved that minimizers exist
provided that $c>0$ is small and
$p \in (2,3)$. In \cite{BS1} the case $p \in (3, \frac{10}{3})$ is considered and a minimizer is obtained for $c >0$ large enough.\medskip

In this paper we consider the case $p \in (\frac{10}{3},6)$. For this range of power the functional $F(u)$ is no more bounded from below on $S(c)$. We shall prove however that it has a mountain pass geometry. 
\begin{definition}\label{defMP}
Given $c>0$, we say that $F(u)$ has a mountain pass geometry on $S(c)$ if there exists $K_c>0$, such that
$$\gamma(c)=\inf_{g \in \Gamma_c } \max_{t \in [0,1]}F(g(t))>\max \{F(g(0)),F(g(1))\},$$
holds in the set
$$\Gamma_c =\{g \in C([0,1],S(c)), \ g(0) \in A_{K_c},F(g(1))<0\},$$
where $A_{K_c}= \{u \in S(c):\ \left \| \triangledown u \right \|_{L^{2}(\mathbb{R}^3)}^2 \leq K_c\}$.
\end{definition}

In order to find critical points of $F(u)$ on $S(c)$ we look at the mountain pass level $\gamma(c)$.
Our main result concerning the existence of solutions of \eqref{eq} is given by the following
\begin{thm}\label{mainvero}
Let $p \in (\frac{10}{3},6)$ and $c>0$ then $F(u)$ has a mountain pass geometry on $S(c)$. Moreover there exists $c_0>0$ such that for any $c \in (0,c_0)$
there exists a couple $(u_{c}, \lambda_c)\in H^1(\R^3)\times \R^-$ solution of \eqref{eq} with $||u_c||_2^2=c$ and $F(u_c)=\gamma(c)$.
\end{thm}

Let us underline some of the difficulties that arise in the study of the existence of critical points for our functional on $S(c)$. First the mountain pass geometry does not guarantee the existence of a bounded Palais-Smale sequence. To overcome this difficulty we introduce the functional

$$ Q(u):=\int_{\R^{3}} |\nabla u|^{2}dx  +\frac{1}{4}\int_{\mathbb{R}^3}\int_{\mathbb{R}^3}\frac{\left | u(x) \right |^2\left | u(y) \right |^2}{\left | x-y \right |}dxdy -\frac{3(p-2)}{2p}\int_{\R^{3}} |u|^{p}dx,$$
the set
$$V(c):=\{u \in S(c)\ : \ Q(u)=0 \}$$
and we first prove that
\begin{equation} \label{mini}
\gamma(c) = \inf_{u \in V(c)}F(u).
\end{equation}
We also show that each constrained critical point of $F(u)$ must lie in $V(c)$. At this point taking advantage of the nice ``shape" of some sequence of paths $(g_n) \subset \Gamma_c$ such that
$$\max_{t \in [0,1]}F(g_n(t)) \to \gamma(c),$$
we construct a special Palais-Smale sequence $\{u_n\} \subset H^1(\R^3)$ at the level $\gamma(c)$ which concentrates around $V(c)$. This localization leads to its boundedness but also provide the information that $Q(u_n) = o(1)$. This last property is crucially used in the study of the compactness of the sequence. Next,  since we look for solutions with a prescribed $L^2$-norm, we must deal with a possible lack of compactness for sequences which does not minimize $F(u)$ on $S(c)$. In our setting it does not seem possible to reduce the problem to the classical vanishing-dichotomy-compactness scenario and to the check of the associated strict subadditivity inequalities, see \cite{L}. To overcome this difficulty we first study the behaviour of the function $c \to \gamma(c)$. The theorem below summarizes its properties.

\begin{thm}\label{maingamma}
Let $p \in (\frac{10}{3},6)$ and for any $c>0$ let $\gamma(c)$ be the mountain pass level. Then
\begin{itemize}
\item[(i)] $ c \to \gamma(c)$ is continuous at each $c >0$.
\item[(ii)] $\displaystyle c \to \gamma(c)$ is non-increasing.
\item[(iii)] There exists $c_0>0$ such that in $(0,c_0)$ the function $c \to \gamma(c)$ is strictly decreasing.
\item[(iv)] There exists $c_{\infty}>0$ such that for all $c \geq c_{\infty}$ the function $c \to \gamma(c)$ is constant.
\item[(v)] $\displaystyle \lim_{c \to 0} \gamma(c) = + \infty$ and $\displaystyle \lim_{c \to \infty}\gamma(c) := \gamma(\infty) >0.$
\end{itemize}
\end{thm}

We show that if $\gamma(c) < \gamma(c_1),$ for all $ c_1 \in (0, c)$ then there exists $u_c \in H^1(\R^3)$ such that $||u_{c}||_2^2 = c$ and $F(u_c) = \gamma(c)$. However we are only able to prove this for $c>0$ sufficiently small. For the other values of $c >0$ the information that $c \to \gamma(c)$ is non increasing permits to reduce the problem of convergence to the one of showing that the associated Lagrange multiplier $\lambda_c \in \R$ is non zero. However we do prove that $\lambda_c=0$ holds for any  $c>0$ is sufficiently large. In view of this property we conjecture that $\gamma(c)$ is not a critical value for $c>0$ large enough. See Remark \ref{versnonexits} in that direction.

\begin{remark}\label{proofnonin}
The proof that $c \to \gamma(c)$ is non increasing is not derived through the use of some scaling. Due to the presence of three terms in $F(u)$ which scale differently such an approach seems difficult. Instead we show that if one adds in a suitable way $L^2$-norm in $\R^3$ then this does not increase the mountain pass level. This approach is reminiscent of the one developed in \cite{JESQ} but here the fact that we deal with a function defined by a mountain pass instead of a global minimum and that $F(u)$ has a nonlocal term makes the proof more delicate.
\end{remark}

To show Theorem \ref{maingamma} (iv) and that $\gamma(c)\to \gamma(\infty) >0$ as $c\to \infty$ in (v) we take advantage of some results of \cite{IARU}. In \cite{IARU} the equation
\begin{equation}\label{eqfree}
-\Delta v + (\left | x \right |^{-1}\ast \left | v \right |^2) v-|v|^{p-2}v=0 \ \mbox{ in } \  \mathbb{R}^{3}
\end{equation}
is considered. Real solutions of (\ref{eqfree}) are searched in the space
\begin{eqnarray}\label{set}
E := \{ u \in D^{1,2}(\R^3) : \int_{\R^3}\int_{\R^3} \frac{|u(x)|^2|u(y)|^2}{|x-y|}dxdy < \infty \}
\end{eqnarray}
which contains $H^1(\R^3)$. This space is the natural space when $\lambda =0$ in (\ref{eq}). In \cite{IARU} it is shown that $F(u)$ defined in $E$ possess a ground state. It is also proved, see Theorem 6.1 of \cite{IARU}, that any real radial solution of (\ref{eqfree}) decreases exponentially at infinity. We extend here this result to any real solution of (\ref{eqfree}). More precisely we prove
\begin{thm}\label{th3.1}
Let $p \in (3,6)$ and  $( u,\lambda )\in E\times \R$ with $\lambda\leq 0$ be a real solution of (\ref{eq}). Then  there exists constants $C_1>0$, $C_2>0$ and $R>0$ such that
\begin{eqnarray}\label{decay}
|u(x)|\leq C_1 |x|^{-\frac{3}{4}}e^{-C_2\sqrt{|x|}},\ \forall\ |x|>R.
\end{eqnarray}
In particular, $u \in H^{1}(\R^{3})$.
\end{thm}

\begin{remark}\label{zeromass}
Clearly the difficult case here is when $\lambda =0$ and it correspond to the so-called {\it zero mass case,} see \cite{BeLi}. This part of Theorem \ref{th3.1} was kindly provided to us by L. Dupaigne \cite{Du}. We point out that the exponential decay when $\lambda =0$ is due to the fact that the nonlocal term is sufficienty strong at infinity. Actually we prove that $(|x|^{-1}* |v|^2) \geq C |x|^{-1}$ for some $C>0$ and $|x|$ large. In contrast we recall that for the equation
\begin{equation}\label{potential}
- \Delta u + V(x) u - |u|^{p-2}u =0, \quad x \in H^1(\R^3),
\end{equation}
if we assume that
$\limsup_{|x| \to \infty}V(x) |x|^{2+ \delta} = 0$
for some $\delta >0$, then positive solutions of (\ref{potential}) decay no faster than $|x|^{-1}$. This can be seen by comparing with an explicit subsolution at infinity $|x|^{- 1}(1+ |x|^{- \delta})$ of $- \Delta + V$.
\end{remark}

Theorem \ref{th3.1} is interesting for itself and also it answers a conjecture of \cite{IARU}, see Remark 6.2 there. For our study the information that any solution of (\ref{eqfree}) belongs to $L^2(\R^3)$ is crucial to derive Theorem \ref{maingamma} (iv)-(v) and the exponential decay  is also used later to prove that our solutions correspond to standing waves unstable by blow-up. \medskip

The phenomena described in Theorems  \ref{mainvero} and \ref{maingamma} are also due to the nonlocal term as we can see by comparing \ref{main} with the classical nonlinear Schr\"odinger equation
\begin{equation}\label{evolution2}
i\psi_{t}+ \Delta \psi  +|\psi|^{p-2}\psi=0 \ \  \text{ in } \R^{3}.
\end{equation}
In \cite{LJ} the existence of standing waves on $S(c)$ when the functional is unbounded from below was considered and a solution obtained for any $c>0$. Here we show in addition that the mountain pass value $\tilde{\gamma}(c)$ associated to (\ref{evolution2}) is strictly decreasing as a function of $c >0$ and that $\tilde \gamma(c) \to 0 $ as $c \to \infty.$ \medskip

The fact that (\ref{mini}) holds and that any constrained critical point of $F(u)$ lies in $V(c)$ implies that the solutions found in Theorem \ref{mainvero} can be considered as ground-states within the solutions having the same $L^2$-norm.

Let us denote the set of minimizers of $F(u)$ on $V(c)$ as
\begin{eqnarray}\label{minimizerset}
\mathcal{M}_c := \{u_c\in V(c)\ : \ F(u_c)=\inf_{u\in V(c)}F(u)\}.
\end{eqnarray}

\begin{thm}\label{naturalconstraint}
Let $p \in (\frac{10}{3}, 6)$ and $c>0$. For each $u_c \in \mathcal{M}_c$ there exists a $\lambda_c \leq 0$ such that $(u_c, \lambda_c) \in H^1(\R^3) \times \R$ solves (\ref{eq}).
\end{thm}

Clearly to prove Theorem \ref{naturalconstraint} we need to show that any minimizer of $F(u)$ on $V(c)$ is a critical point of $F(u)$ restricted to $S(c)$, namely that $V(c)$ acts as a natural constraint.
As additional properties of elements of $\mathcal{M}_c$ we have :

\begin{lem}\label{description}
Let $p \in (\frac{10}{3},6)$ and $c>0$ be arbitrary. Then
\begin{enumerate}
  \item [(i)] If $u_c \in \mathcal{M}_c$ then also $|u_c| \in \mathcal{M}_c$ .
  \item [(ii)] Any minimizer $u_c \in\mathcal{M}_c$ has the form $e^{i\theta}|u_c|$ for some $\theta \in \mathbb{S}^1$ and $|u_c(x)| >0 $ a.e. on $\R^3$.
\end{enumerate}
\end{lem}

In view of Lemma \ref{description} each elements of $\mathcal{M}_c$ is a real positive function multiply by a constant complex factor.

Concerning the dynamics we first consider the question of global existence of solutions for the Cauchy problem.  In the case $p \in (2, \frac{10}{3})$ global existence in time is guaranteed for initial data in $H^1(\R^3)$, see for instance \cite{TC}. In the case $p \in (2, \frac{10}{3})$ the standing waves found in \cite{BS1}, \cite{BS}, \cite{SS} by minimization are orbitally stable. This is proved following the approach of Cazenave-Lions \cite{CL}.
In the case $p \in (\frac{10}{3}, 6)$ the global existence in time of solutions for the Cauchy problem associated to \eqref{main} does not hold for arbitrary initial condition. However we are able to prove the following global existence result.

\begin{thm}\label{global}
Let $p \in (\frac{10}{3},6)$ and $u_0 \in H^1(\mathbb{R}^{3},\mathbb{C})$ be an initial condition associated to \eqref{main} with $c=||u_0||_2^2$. If
$$Q(u_0)>0 \text{  and } F(u_0)<\gamma(c),$$
then  the solution of \eqref{main} with initial condition $u_0$ exists globally in times.
\end{thm}

In Remark \ref{rem1.3'}  we prove that the set
$$\mathcal{O}= \{u_0\in S(c)\ :\ Q(u_0)>0 \mbox{ and } F(u_0)<\gamma(c)  \}$$
is not empty.

Next we prove that the standing waves corresponding to elements of $\mathcal{M}_c$ are unstable in the following sense.
\begin{definition}
A standing wave $e^{i\omega t}v(x)$ is strongly unstable if for any $\varepsilon >0$ there exists $u_0 \in H^1(\mathbb{R}^{3},\mathbb{C})$ such that $\left \| u_0-v \right \|_{H^1}< \varepsilon$ and the solution $u(t,\cdot)$ of the equation  \eqref{main} with $u(0, \cdot)=u_0$ blows up in a finite time.
\end{definition}

\begin{thm}\label{th2.1}
Let $p \in (\frac{10}{3}, 6)$ and $c>0$. For each $u_c \in \mathcal{M}_c$ the standing wave $e^{-i \lambda_c t}u_c$ of (\ref{main}), where $\lambda_c \in \R $ is the Lagrange multiplier, is strongly unstable.
\end{thm}

\begin{remark}
The proof of Theorem \ref{th2.1} borrows elements of the original approach of Berestycki and Cazenave \cite{BECA}. The starting point is the variational characterization of  $u_c \in \mathcal{M}_c $  and the decay estimates established in Theorem \ref{th3.1} proves crucial to use the virial identity.
\end{remark}

\begin{remark}
For previous results concerning the instability of standing waves of (\ref{main}) we refer to \cite{HK} (see also \cite{HK2}). In \cite{HK}, working in the subspace of radially symmetric functions, it is proved that for $\lambda < 0$ fixed and $p \in (\frac{10}{3}, 6)$ the equation (\ref{eq}) admits a ground state which is strongly unstable. However when we work in all $H^1(\R^3)$ it is still not known if ground states, or at least one of them, are radially symmetric. In that direction we are only aware of the result of \cite{GEPRVI} which gives a positive answer when $p \in (2,3)$ and for $c>0$ sufficiently small. In this range the critical point is found as a minimizer of $F(u)$ on $S(c)$.
\end{remark}

Finally we prove
\begin{thm}\label{cor1.1}
Let $p \in (\frac{10}{3},6)$. Any ground state of (\ref{eqfree}) is strongly unstable.
\end{thm}

\begin{remark}
In the {\it zero mass case} there seems to be few results of stability/instability of standing waves. We are only aware of \cite{KaOh} for a stability result.
\end{remark}

The paper is organized as follows. In Section \ref{Section2} we establish the mountain pass geometry of $F(u)$ on $S(c)$. In Section \ref{Section3} we construct the special bounded Palais-Smale sequence at the level $\gamma(c)$. In Section \ref{Section4} we show the convergence of the Palais-Smale sequence and we conclude the proof of Theorem \ref{mainvero}. In Section \ref{Section5} some parts of Theorem \ref{maingamma} are established. In Section \ref{Section51} we prove
Theorem \ref{naturalconstraint} and Lemma \ref{description}. In Section \ref{Section6} we prove Theorem \ref{th3.1} and using elements from \cite{IARU} we end the proof of Theorem \ref{maingamma}.  Section \ref{Section8} is devoted to the proof of Theorems \ref{global}, \ref{th2.1} and \ref{cor1.1}. Finally in Section \ref{Section7} we discuss the nonlinear Schr\" odinger equation case. \medskip

\textbf{Acknowledgement:} The authors thanks Professor Louis Dupaigne for providing to them a proof of Theorem \ref{th3.1} in the case $\lambda =0$. We also thanks Professor Masahito Ohta for pointing to us the interest of studying the stability/instability of the ground states of (\ref{eqfree}).

% We say that the standing wave $e^{i\omega t}v(x)$ is stable if for any $\varepsilon >0$, there is $\delta >0$ such that if $u_0 \in H^1(\mathbb{R}^{3},\mathbb{C})$ satisfies
% $$\inf_{\theta \in \mathbb{R},y \in \mathbb{R}^3}\left \| u_0-e^{i\theta}v(\cdot +y) \right \|_{H^1}<\delta,$$
% then we have
% $$\sup_{t > 0}\inf_{\theta \in \mathbb{R},y \in \mathbb{R}^3}\left \| u(t)-e^{i\theta}v(\cdot +y) \right \|_{H^1}< \varepsilon .$$
% where $u(t,\cdot)$ is the solution of equation (2.1) with the initial datum $u(0, \cdot)=u_0$.

% Otherwise, the standing wave $e^{i\omega t}v(x)$ is said to be orbitally unstable. \\

\subsection{Notations}
In the paper it is understood that all functions,
unless otherwise stated, are
complex-valued, but for simplicity we write $L^{s}(\R^{3}), H^{1}(\R^{3}) ....$,  and for any $1\leq s < +\infty$, $L^s(\R^3)$ is the usual Lebesgue space endowed
with the norm
$$\|u\|_{s}^s:=\int_{\R^{3}} |u|^sdx,$$
%We denote $D^{1,2}(\R^3)$ the completition of $C^{\infty}_{0}$ with respect to the norm
and
%$$||u||_{D^{1,2}}^2:=\int_{\R^{3}} |\nabla u|^2dx$$
$H^1(\R^3)$ the usual Sobolev space  endowed with the norm
$$\|u\|^2:=\int_{\R^{3}} |\nabla u|^2dx+\int_{\R^{3}} |u|^2dx.$$
Moreover we define, for short,
the following quantities
$$A(u):=\int_{\R^{3}} |\nabla u|^{2}dx,\ \ \ B(u):=\int_{\mathbb{R}^3}\int_{\mathbb{R}^3}\frac{\left | u(x) \right |^2\left | u(y) \right |^2}{\left | x-y \right |}dxdy$$
$$ \ C(u):=-\int_{\R^{3}} |u|^{p}dx, \ \ \ D(u):=\int_{\mathbb{R}^3}\left | u\right |^2dx. $$
Then
\begin{equation}\label{}
Q(u)=A(u)+\frac{1}{4}B(u)+\frac{3(p-2)}{2p}C(u ).
\end{equation}

\section{The mountain pass geometry on the constraint}\label{Section2}
In this section, we discuss the Mountain Pass Geometry (``MP Geometry" for short) of the functional $F(u)$ on the $ L^2$-constraint $S(c)$.
We show the following:
% \begin{thm}\label{thm1.1}
% When $p \in (3,\frac{10}{3}]$, there exists $c^{\star}>0$, sufficiently large, such that $F(u)$ has a MP geometry on the constraint $S(c)$ for any $c \in (c^{\star},+\infty)$.
% \end{thm}

\begin{thm}\label{thm1.2}
When $p \in (\frac{10}{3},6)$, for any $c >0$, $F(u)$ has a MP geometry on the constraint $S(c)$.
\end{thm}

Before proving Theorem \ref{thm1.2} we establish some lemmas. We first introduce the Cazenave's scaling \cite{TC}. For $u\in S(c)$, we set $u^t(x)=t^{\frac{3}{2}}u(tx),t>0$, then
\begin{eqnarray*}
A(u^t)=t^2A(u)\ , \ D(u^t)=D(u),
\end{eqnarray*}
and
\begin{eqnarray*}
  B(u^t)= t B(u)\ , \ C(u^t) = t^{\frac{3}{2}(p-2)} C(u).
\end{eqnarray*}
Thus
\begin{eqnarray}\label{1.1}
F(u^t)=\frac{t^2}{2}A(u)+\frac{t}{4}B(u)+\frac{t^{\frac{3}{2}(p-2)}}{p}C(u).
\end{eqnarray}

\begin{lem}\label{lm1.1}
Let $u \in S(c)$, $c >0$ be arbitrary but fixed and $p \in (\frac{10}{3}, 6)$, then:
\begin{enumerate}
 \item   $A(u^t) \to \infty$ and
        $F(u^t)  \to -\infty$, as $t \to \infty$.
 \item  There exists $k_0 >0$ such that $Q(u)>0$ if $||\nabla u||_2\leq k_0$ and $-C(u) \geq k_0$ if $Q(u)=0$.
   \item If $F(u)<0$ then $Q(u)<0.$
\end{enumerate}
\end{lem}
\begin{proof}
We notice that
\begin{equation}\label{sisi2}
F(u)-\frac{2}{3(p-2)}Q(u)=\frac{3p-10}{6(p-2)}A(u)+\frac{3p-8}{12(p-2)}B(u).
\end{equation}
Thus (3) holds since the RHS is always positive. Moreover, thanks to Gagliardo-Nirenberg inequality there exists a constant $C(p)>0$ such that
$$Q(u) \geq A(v)-C(p) A(u)^{\frac{3(p-2)}{4}}D(v)^{\frac{6-p}{4}}.$$
The fact that $\frac{3(p-2)}{4}>1$ insures that $Q(u)>0$ for sufficiently small $A(u)$. Also when $Q(u)=0$
$$ - C(u) = \frac{2p}{3(p-2)} [ A(u) + \frac{1}{4} B(u)] \geq \frac{2p}{3(p-2)} A(u)$$
and this ends the proof of (2). Finally (1) follows directly from (\ref{1.1}) and since $A(u^t) = t^2 A(u).$
\end{proof}
Our next lemma is inspired by Lemma 8.2.5 in \cite{TC}.

\begin{lem}\label{growth}
When $p \in (\frac{10}{3}, 6)$, given $u \in S(c)$ we have: \\
(1)\ There exists a unique $t ^{\star}(u)>0$, such that $u^{t ^{\star}} \in V(c)$;\\
(2)\ The mapping $t \longmapsto F(u^{t})$ is concave on $[t ^{\star}, \infty)$;\\
(3)\ $t ^{\star}(u)<1$ if and only if $Q(u)<0$;\\
(4)\ $t ^{\star}(u)=1$ if and only if $Q(u)=0$;\\
(5)\ $$Q(u^t)\left\{\begin{matrix}
\ >0,\ \forall\ t &\in& (0,t^*(u));\\
\ <0, \ \forall\ t&\in& (t^*(u),+\infty).
\end{matrix}\right.$$
(6)\ $F(u^{t})<F(u^{t ^{\star}})$, for any $t>0$ and $t \neq t ^{\star}$;\\
(7)\ $\frac{\partial}{\partial t} F(u^{t})=\frac{1}{t}Q(u^{t})$, $\forall t >0$.
\end{lem}

\begin{proof}
 Since $$F(u^{t})=\frac{t^2}{2}A(u)+\frac{t}{4}B(u)+\frac{t^{\frac{3}{2}(p-2)}}{p}C(u)$$
we have that
$$ \frac{\partial}{\partial t} F(u^{t}) = t A(u)+ \frac{1}{4}B(u)+\frac{3(p-2)}{2p}t^{\frac{3}{2}(p-2)-1}C(u)
= \frac{1}{t}Q(u^{t})$$
%\begin{eqnarray*}
%\frac{\partial}{\partial t} F(u^{t})&=& t A(u)+ \frac{1}{4}B(u)+\frac{3(p-2)}{2p}t^{\frac{3}{2}(p-2)-1}C(u) \\
%&=& \frac{1}{t}Q(u^{t})
%\end{eqnarray*}
and this proves (7). Now we denote
$$y(t)= t A(u) + \frac{1}{4}B(u)+\frac{3(p-2)}{2p}t^{\frac{3}{2}(p-2)-1}C(u),
$$
and observe that $Q(u^{t})= t \cdot y(t)$. After direct calculations, we  see that:
\begin{eqnarray*}
y'(t)&=& A(u) +\frac{3(p-2)(3p-8)}{4p}t^{\frac{3p-10}{2}}C(u); \\
y''(y)&=&  \frac{3(p-2)(3p-8)}{4p}\cdot \frac{3p-10}{2}\cdot t^{\frac{3p-12}{2}}C(u).
\end{eqnarray*}

>From the expression of $y'(t)$ we know that $y'(t)$ has a unique zero that we denote $t_0>0$. Since $p \in (\frac{10}{3}, 6)$ we see that $y''(t)<0$ and $t_0$ is the unique maximum point of $y(t)$. Thus in particular the function $y(t)$ satisfies:\\
(i)\ $y(t_0)=\max_{t >0}y(t)$;\\
(ii)\ $y(0)=\frac{1}{4}B(u)$;\\
(iii)\ $\lim_{t\to +\infty}y(t)=-\infty$;\\
(iv)\ $y(t)$ decreases strictly in $[t_0, +\infty)$ and increases strictly in $(0, t_0]$.\medskip

Since $B(u) \neq 0$, by the continuity of $y(t)$, we deduce that $y(t)$ has a unique zero $t^{\star}>0$. Then $Q(u^{t^*}) =0$ and point (1) follows. Point (2) (3) and (5) are also easy consequences of (i)-(iv). Since $\frac{\partial}{\partial t} F(u^{t})|_{t=t^{\star}}=0$, $\frac{\partial^2}{\partial t^2} F(u^{t})|_{t=t^{\star}} = y'(t^*)<0$ and $t^{\star}$ is unique we get (4) and (6).
\end{proof}

\begin{proof}[Proof of Theorem  \ref{thm1.2}]
We denote
$$ \alpha_k:=\sup_{u \in C_k} F(u) \quad \mbox{ and } \quad \beta_k:=\inf_{u \in C_k} F(u)$$
where 
$$ C_k:=\{u \in S(c): A(u)=k,k>0\}.$$
Let us show that there exist $0<k_1<k_2$ such that
\begin{equation}\label{dadim}
\alpha_{k}< \beta_{k_2}  \text{ for all } k \in (0,k_1] \text{ and } Q(u)>0 \text{ if } A(u)<k_2.
\end{equation}
Notice that, from Hardy-Littlewood-Sobolev's inequality and Gagliardo-Nirenberg's inequalities, it follows that
\begin{eqnarray*}
F(u)&\leq& \frac{1}{2}A(u)+\frac{1}{4}B(u) \leq  \frac{1}{2}A(u)+ C(p)\|u\|_{L^{\frac{12}{5}}}^4 \nonumber \\
&\leq& \frac{1}{2}A(u)+ \widetilde{C}(p)A(u)^{\frac{1}{2}}\cdot D(u)^{\frac{3}{2}}.
\end{eqnarray*}
In particular $\alpha_{k_1} \to 0^+$ as $k_1 \to 0^+$. On the other hand still by the Gagliardo-Nirenberg inequality we have
\begin{equation*}
F(u)\geq \frac{1}{2}A(u)+\frac{1}{p}C(u)
\geq \frac{1}{2}A(u)- C(p)A(u)^{\frac{3(p-2)}{4}}\cdot D(u)^{\frac{6-p}{4}}.
\end{equation*}
Thus, since $\frac{3(p-2)}{2} >1$, $\beta_{k_2} \geq \frac{1}{4} k_2$ for any $k_2 >0$ small enough. These two observations and Lemma 2.1 (2) prove that (\ref{dadim}) hold. We now fix a $k_1 >0$ and a $k_2 >0$ as in (\ref{dadim}). Thus for
\begin{eqnarray*}
\Gamma_c =\{g \in C([0,1],S(c)), \ g(0) \in A_{k_1},F(g(1))<0\},
\end{eqnarray*}
if $\Gamma_c \neq \emptyset$, then from the definition of $\gamma(c)$, we have $\gamma(c)\geq \beta_{k_2}>0$
We only need to verify that $\Gamma_c \neq \emptyset$. This fact follows from
Lemma  \ref{lm1.1} (1).
\end{proof}

\begin{remark}\label{allok}
As it is clear from the proof of Theorem \ref{thm1.2} we can assume without restriction that
$$\sup_{u \in A_{K_c}}F(u) < \gamma(c)/2$$
where $A_{K_c}$ is introduced in the Definition \ref{defMP}.
\end{remark}

\begin{lem}\label{mpestimate}
When $p \in (\frac{10}{3},6)$, we have
$$\gamma (c)=\inf_{u \in V(c)}F(u).\\ $$
\end{lem}

\begin{proof}
Let  us argue by contradiction. Suppose there exists $v \in V(c)$ such that $F(v)< \gamma (c)$,
and let, for $\lambda >0$,
$$v^{\lambda}(x) = \lambda^{3/2}v(\lambda x).$$
Then, since $A(v^{\lambda}) = \lambda^2 A(v)$ there exists $0 < \lambda_1 < 1$ sufficiently small so that $v^{\lambda_1} \in A_{k_1}$. Also by Lemma \ref{lm1.1} (1) there exists a $\lambda_2 >1$ sufficiently large so that $F(v^{\lambda_2})<0$. Therefore if we define
$$g(t) = v^{(1-t)\lambda_1 + t \lambda_2}, \quad \mbox{for } t \in [0,1]$$
we obtain a path in $\Gamma_c$. By definition of $\gamma(c)$ and using Lemma \ref{growth},
$$\gamma(c)\leq \max_{t \in [0,1]}F(g(t)) = F\Big(g(\frac{1- \lambda_1}{\lambda_2 - \lambda_1})\Big) =  F(v),$$
and thus $$\gamma (c) \leq \inf_{u \in V(c)}F(u).$$
On other hand thanks to Lemma \ref{lm1.1} any path in $\Gamma_c$ crosses $V(c)$ and hence
$$\max_{t\in [0,1]}F(g(t))\geq \inf_{u\in V(c)}F(u).$$
\end{proof}

\section{Localization of a PS sequence}\label{Section3}

In this section we prove a localization lemma for a specific Palais-Smale sequence $\{u_n\} \subset S(c)$ for $F(u)$ constrained to $S(c)$. From this localization we deduce that the sequence is bounded and that $Q(u_n)= o(1)$. This last property will be essential later to establish the compactness of the sequence.
First we observe that, for any fixed $c>0$, the set
$$L:= \{ u \in V(c), F(u) \leq \gamma(c) +1 \}$$
is bounded. This follows directly from the observation that
\begin{equation}\label{keyest}
F(u)-\frac{2}{3(p-2)}Q(u)=\frac{3p-10}{6(p-2)}A(u)+\frac{3p-8}{12(p-2)}B(u)
\end{equation}
and the fact that  $\frac{3p-10}{6(p-2)}>0$, $\frac{3p-8}{12(p-2)}>0$ if $p \in (\frac{10}{3},6)$. \medskip

Let $R_0 >0$ be such that $L \subset B(0, R_0)$ where $B(0, R_0):= \{u \in H^1(\R^3), ||u||\leq R_0 \}.$ \medskip

The crucial localization result is the following.
\begin{lem}\label{Forse}
Let $p \in (\frac{10}{3},6)$ and
\begin{eqnarray*}
K_{\mu} := \left\{ u \in S(c) \text{ s.t. } |F(u) -\gamma(c)|\leq \mu, \ dist(u, V(c))\leq 2\mu, \
||F^{'}|_{S(c)}(u)||_{H^{-1}} \leq 2 \mu \right\},
\end{eqnarray*}
then for any $\mu>0$, the set $K_{\mu} \bigcap B(0, 3 R_0)$ is not empty.
\end{lem}

In order to prove Lemma \ref{Forse} we need to develop a deformation argument on $S(c)$. Following \cite{BELI} we recall that, for any $c>0$, $S(c)$ is a submanifold of $H^1(\R^3)$ with codimension 1 and the tangent space at a point $\bar u \in S(c)$ is defined as
$$T_{\bar u}S(c)=\{ v \in H^1(\R^3)  \text{ s.t. } (\bar u, v)_2=0 \}.$$
The restriction $F_{|_{S(c)}}:S(c) \rightarrow \R$
is a $C^1$ functional on $S(c)$ and for any $\bar u \in S(c)$ and any $v\in T_{\bar u}S(c)$
$$\langle F_{|_{S(c)}}'(\bar u), v \rangle=\langle F'(\bar u), v \rangle.$$
We use the notation $||dF_{|_{S(c)}}(\bar u)||$ to indicate the norm
in the cotangent space $T_{\bar u}S(c)'$, i.e the dual norm induced by the
norm of $T_{\bar u}S(c)$, i.e
$$||dF_{|_{S(c)}}(\bar u)||:=\sup_{||v|| \leq 1, \ v \in T_{\bar u}S(c)} |\langle dF(\bar u), v\rangle | .$$
Let $\tilde{S}(c):= \{ u \in S(c) \text{ s.t. } dF{|_{S(c)}}(u) \neq 0\}.$ We know from \cite{BELI} that there exists a locally Lipschitz pseudo
gradient vector field  $Y\in {\mathcal C}^1(\tilde{S}(c),T(S(c))$ ( here $T(S(c))$ is the tangent bundle) such that
\begin{equation}\label{prop1}
\|Y(u)\|\leq 2 \, ||dF_{|_{S(c)}}(u)|| \hbox{, }
\end{equation}
and \begin{equation}\label{finaps}
\langle F_{|_{S(c)}}'(\bar u), Y(u)\rangle \, \geq  ||dF_{|_{S(c)}}(u)||^2 \hbox{, }
\end{equation}
for any $ u \in \tilde{S}(c)$. Note that $||Y(u)|| \neq 0$ for $u \in \tilde{S}(c)$ thanks to (\ref{finaps}).
Now for an arbitrary but fixed $ \mu>0 $ we consider the sets
$$\tilde {N}_{\mu}:=\{ u \in S(c) \text{ s.t. } |F(u)-\gamma(c)|\leq \mu, \ dist(u, V(c)) \leq 2\mu, \ ||Y(u)|| \geq 2 \mu\}$$
$$N_{\mu}:=\{ u \in S(c) \text{ s.t.}\  |F(u)-\gamma(c)| < 2\mu \}$$ 
where, for a subset $\mathcal{A}$ of $S(c)$, $dist(x,{\mathcal A}):=\inf \{ ||x-y|| :  y \in {\mathcal A}\}$.
Assuming that $\tilde {N}_{\mu}$ is non empty there exists a locally Lipschitz function $g: S(c) \rightarrow [0,1]$ such that
\begin{equation*}
g=\left\{
\begin{array}{ll}
1 \text{ on } \tilde{N}_{\mu} \\
0  \text{ on }  N^c_{\mu}.
\end{array}
\right.
\end{equation*}
We also define on $S(c)$ the vector field $W$  by
\begin{equation}\label{VF}
W(u) = 
\left\{
\begin{array}{ll}
& - g (u) \frac{Y( u)}{||Y(u)||}   \text{ if } u \in \tilde{S}(c)\\
& 0  \text{ if } u \in S(c) \backslash \tilde{S}(c)
\end{array}\right.
\end{equation}
and the pseudo gradient flow 
\begin{equation}\label{Flow}
\left\{
\begin{array}{ll}
& \frac{d}{dt}\eta(t, u)= W (\eta(t,  u))\\
& \eta(0, u)=u.
\end{array}\right.
\end{equation}
The existence of a unique solution $\eta(t, \cdot)$ of \eqref{Flow} defined for all $t \in \R$ follows from standard arguments and we refer to  Lemma 5 in \cite{BELI} for this. Let us recall some of its basic properties that will be useful to us
\begin{itemize}
\item $\eta(t, \cdot)$ is a homeomorphism of $S(c)$;
\item $\eta(t, u)=u$ for all $t\in \mathbb{R}$ if $|F(u)-\gamma(c)|\geq 2\mu$;
\item $\frac{d}{dt} F (\eta (t, u))=\langle dF(\eta (t, u)), 
W(\eta(t,u)) \rangle \leq 0$ for all  $t \in \R$ and $u\in S(c).$
\end{itemize}

\begin{proof}[Proof of Lemma \ref{Forse}]:
Let us define, for $\mu >0$,
$$\Lambda_{\mu}=\left\{ u \in S(c) \text{ s.t. } |F(u) -\gamma(c)|\leq \mu, \ dist(u, V(c))\leq 2\mu \right\}.$$
In order to prove  Lemma \ref{Forse} we argue by contradiction  assuming that there exists $\bar \mu \in (0, \gamma(c)/4)$ such that
\begin{equation}\label{contradiction}
u \in \Lambda_{\bar \mu} \cap B(0, 3R_0) \, \Longrightarrow  \, \ ||F^{'}|_{S(c)}(u)||_{H^{-1}} > 2 \bar \mu.
\end{equation}
Then it follows from \eqref{finaps} that
\begin{equation}\label{contradiction2}
 u \in  \Lambda_{\bar \mu} \cap B(0, 3R_0) \, \Longrightarrow u \in  \tilde N_{\bar \mu}.
\end{equation}
Also notice that, since by (\ref{Flow}),
$$||\frac{d}{dt}\eta(t,u)|| \leq 1, \quad \forall t \geq 0, \, \forall u \in S(c),$$
there exists $s_0>0$ depending on $\bar \mu >0$ such that, for all $s \in (0,s_0)$,
\begin{equation}\label{fondamentale}
u \in \Lambda_{\frac{\bar \mu}{2}} \cap B(0, 2R_0) \ \Longrightarrow \, \eta(s,u) \in B(0, 3R_0) \mbox{ and } dist(\eta(s,u) , V(c))\leq 2\bar \mu.
\end{equation}
We claim that, taking $\varepsilon >0$ sufficiently small, we can construct a path $g_{\varepsilon}(t) \in \Gamma_c$ such that
$$ \max_{t \in [0,1]} F(g_{\varepsilon}(t) \leq \gamma(c) + \varepsilon$$
and
\begin{equation}\label{imn}
F(g_{\varepsilon}(t)) \geq \gamma(c) \Longrightarrow \ g_{\varepsilon}(t) \in \Lambda_{\frac{\bar \mu}{2}} \cap B(0, 2R_0).
\end{equation}
Indeed, for $\varepsilon >0$ small, let $u \in V(c)$ be such that $F(u) \leq \gamma(c) +\varepsilon$
and consider the path defined in Lemma \ref{mpestimate} by
\begin{eqnarray}\label{path11}
g_{\varepsilon}(t)= u^{(1-t)\lambda_1 + t \lambda_2}, \quad \mbox{for } t \in [0,1].
\end{eqnarray}
Clearly
$$\max_{t \in [0,1]} F(g_{\varepsilon}(t)) \leq \gamma(c) + \varepsilon.$$ Also for $t_{\varepsilon}^* >0$ such that $(1- t_{\varepsilon}^*) \lambda_1 + t_{\varepsilon}^* \lambda_2 = 1$ we have, since $g_{\varepsilon}(t_{\varepsilon}^*) \in V(c)$, that
\begin{equation}\label{azzollini}
\frac{d^2}{d^2s}F(g_{\varepsilon}(s))_{|t_{\varepsilon}^*} = - \frac{1}{4}B(u) - \frac{3}{2p} (p-2) (5- \frac{3}{2}p) C(u) \leq - C k_0 < 0
\end{equation}
where $k_0 >0$ is given in Lemma \ref{lm1.1} (2). The estimate (\ref{azzollini}) is uniform with respect to the choice of $\varepsilon >0$ and of $u \in V(c)$. Thus, by Taylor's formula, it is readily seen that
$$ \{ t \in [0,1] : F(g_{\varepsilon}(t)) \geq \gamma(c) \} \subset  [t_{\varepsilon}^* - \alpha_{\varepsilon}, t_{\varepsilon}^* + \alpha_{\varepsilon}]$$
for some $\alpha_{\varepsilon} >0$ with $\alpha_{\varepsilon} \to 0$ as $\varepsilon \to 0$. The claim (\ref{path11}) follows for continuity arguments.
\medskip

We fix a $  \varepsilon \in (0, \frac{1}{4}\bar\mu s_0)$ such that (\ref{imn}) hold. Applying the pseudo gradient flow, constructed with $\bar \mu >0$, on $g_{\varepsilon}(t)$ we see that
$\eta(s, g_{\varepsilon}(\cdot)) \in \Gamma_c$ for all $s>0.$ Indeed $\eta(s, u)=u$ for all $s>0$ if $|F(u)- \gamma(c)| \geq 2 \bar{\mu}$ and we conclude by Remark \ref{allok}.

We claim that taking  $ \displaystyle s^* := \frac{4 \varepsilon}{\bar{\mu}}  < s_0$ 
\begin{equation}\label{ok}
\max_{t \in [0,1]}F(\eta(s^*,g_{\varepsilon}(t)))< \gamma(c).
\end{equation}
If (\ref{ok}) hold we have a contradiction with the definition of $\gamma(c)$ and thus the lemma is proved. To prove (\ref{ok}) for simplicity we set
$w=g_{\varepsilon}(t)$ where $t \in [0,1].$ If $F(w)< \gamma(c)$ there is nothing to prove since then $F(\eta(s^*, w) \leq F(w)< \gamma(c)$ for any $s >0$. If $F( w) \geq \gamma(c) $ we assume by contradiction that $F(\eta(s, w)) \geq \gamma(c) $ for all $s \in [0, s^*]$. Then by \eqref{fondamentale} and \eqref{imn},
$\eta(s, w) \in \Lambda_{\bar \mu} \cap B(0, 3R_0)  $ for all $s \in[0,s^*]$.  In particular $||Y(\eta(s, w))||\geq 2\bar \mu$ and $g(\eta(s,w))=1$ for all $s \in [0, s^*]$. Thus 
$$\frac{d}{ds}F(\eta(s, w))=\langle dF(\eta (s,  w)),  - \frac{Y( \eta(t,u))}{||Y(\eta(t,u))||}\rangle.$$
By integration, and since $\displaystyle s^* = \frac{4 \varepsilon}{\bar \mu},$ we get
$$F(\eta(s^*, w))\leq F(w)- \bar{\mu} s^* \leq (\gamma(c)+\varepsilon)- 2 \varepsilon < \gamma(c)-\varepsilon.$$
This proves the claim (\ref{ok}) and the lemma.
\end{proof}

\begin{lem}\label{apriori}
Let $p \in (\frac{10}{3}, 6)$, then there exists a sequence $\{u_n\} \subset S(c)$ and a constant $\alpha>0$ fulfilling
$$Q(u_n)=o(1), \  \ F(u_n)=\gamma(c)+o(1),$$
$$ ||F^{'}|_{S(c)}(u_n)||_{H^{-1}}=o(1) , \ ||u_n||\leq \alpha.$$
\end{lem}
\begin{proof}
First let us consider $\{u_n\} \subset S(c)$  such that $\{u_n\} \subset B(0, 3R_0)$,
$$dist(u_n, V(c))=o(1), \ |F(u_n) -\gamma(c)| = o(1), \ ||F^{'}|_{S(c)}(u_n)||_{H^{-1}}=o(1).
$$
Such sequence exists thanks to Lemma \ref{Forse}. To prove the lemma we just have to show  that
$Q(u_n)=o(1)$. It is readily checked that $||dQ(\cdot)||_{H^{-1}}$ is bounded on any bounded set of $H^1(\R^3)$ and thus in particular on $B(0, 3R_0)$. Now, for any $n \in \N$ and any $w \in V(c)$ we can write
$$Q(u_n) = Q(w)  + dQ(au_n + (1-a)w) (u_n-w)$$
where $a \in [0,1]$. Thus since $Q(w)=0$ we have
\begin{equation}\label{fin}
|Q(u_n)| \leq \max_{u \in B(0, 3R_0)}||dQ||_{H^{-1}}||u_n- w||.
\end{equation}
Finally choosing $\{w_m\} \subset V(c)$ such that
$$||u_n - w_m|| \to dist(u_n, V(c)) \mbox{ as } m \to \infty,$$
since $dist(u_n, V(c)) \to 0$ we obtain from (\ref{fin}) that $Q(u_n) = o(1)$.
\end{proof}

\section{Compactness of our Palais-Smale sequence}\label{Section4}
\begin{prop}\label{prop1.3}
Let $\{v_n\} \subset S(c)$ be a bounded Palais-Smale for $F(u)$ restricted to $S(c)$ such that $F(v_n) \to \gamma(c)$. Then there is a sequence $\{\lambda_n\}\subset \mathbb{R}$, such that, up to a subsequence:\\
(1)\ $v_n \rightharpoonup v_c$ weakly in $H^1(\mathbb{R}^3)$;\\
(2)\ $\lambda_n \to \lambda_c$ in $\mathbb{R}$;\\
(3)\ $-\Delta v_n-\lambda_n v_n+(\left | x \right |^{-1}\ast \left | v_n \right |^2) v_n-|v_n|^{p-2}v_n \to 0$ in $H^{-1}(\mathbb{R}^3)$;\\
(4)\ $-\Delta v_n-\lambda_c v_n+(\left | x \right |^{-1}\ast \left | v_n \right |^2) v_n-|v_n|^{p-2}v_n \to 0$ in $H^{-1}(\mathbb{R}^3)$;\\
(5) $-\Delta v_c-\lambda_c v_c+(\left | x \right |^{-1}\ast \left | v_c \right |^2) v_c-|v_c|^{p-2}v_c= 0$ in $H^{-1}(\mathbb{R}^3).$
\end{prop}

\begin{proof}
Point (1) is trivial. Since $\{v_n\} \subset H^1(\mathbb{R}^3)$ is bounded, following Berestycki and Lions (see Lemma 3 in \cite{BELI}), we know that: \\
\begin{eqnarray*}
F'|_{S(c)}(v_n) &\longrightarrow& 0 \ \mbox{ in } H^{-1}(\mathbb{R}^3) \\
&\Longleftrightarrow & F'(v_n)-\langle F'(v_n),v_n \rangle v_n \longrightarrow 0 \mbox{ in } H^{-1}(\mathbb{R}^3).\\
\end{eqnarray*}
Thus, for any $w\in H^1(\mathbb{R}^3)$,
\begin{eqnarray*}
\langle F'(v_n)-\langle F'(v_n),v_n \rangle v_n, w  \rangle =\int_{\mathbb{R}^3}\nabla v_n  \nabla w dx + \int_{\mathbb{R}^3}\int_{\mathbb{R}^3}\frac{|v_n(x)|^2}{|x-y|}v_n(y) w(y)dxdy \\
-\int_{\mathbb{R}^3}|v_n|^{p-2}v_nwdx -\lambda_n \int_{\mathbb{R}^3}v_n(x)w(x)dx,
\end{eqnarray*}
with
\begin{equation}\label{lambda}
\lambda_n=\frac{1}{\|v_n\|_2}\Big\{\|\nabla v_n\|_2^2+ \int_{\mathbb{R}^3}\int_{\mathbb{R}^3}\frac{|v_n(x)|^2 v_n(x)^2}{|x-y|}dxdy -\|v_n\|_p^p\Big\}.
\end{equation}
Thus we obtain (3) with $\{\lambda_n \} \subset \R$ defined by (\ref{lambda}). If (2) holds then (4) follows immediately from (3). To prove (2), it is enough to verify that $\{\lambda_n\} \subset \R$ is bounded. But since $\{v_n\} \subset H^1(\R^N)$ is bounded, by the Hardy-Littlewood-Sobolev inequality and Gagliardo-Nirenberg inequality, it is easy to see that all terms in the RHS of (\ref{lambda}) are bounded. Finally we refer to Lemma 2.2 in \cite{ZZ} for a proof of (5).
\end{proof}

\begin{lem}\label{compact}
Let $p \in (\frac{10}{3},6)$ and  $\{u_n\} \subset S(c)$  be a bounded sequence such that
$$ Q(u_n) = o(1) \quad \mbox{and} \quad F(u_n) \rightarrow \gamma(c) \text{ with } \gamma(c)>0,$$
then, up to a subsequence and up to translation $u_n \rightharpoonup \bar u\neq 0$.
\end{lem}
\begin{proof}
If the lemma does not hold it means by standard arguments that $\{u_n\} \subset S(c)$ is vanishing and thus that $C(u_n) = o(1)$ (see \cite{L}). Thus let us argue by contradiction assuming that $C(u_n) = o(1)$,
i.e. that, since $Q(u_n) = o(1)$, $A(u_n)+\frac{1}{4}B(u_n)=o(1)$. Now from \eqref{keyest} we immediately deduce that $F(u_n)=o(1)$ and this contradicts the assumption that $F(u_n) \to \gamma(c) >0.$
% drives to a contradiction. \\
%Thanks to Sobolev embedding there exists two constants $\alpha, \beta>0$ such that
%\begin{equation}\label{allalieb}
%||u||_6\leq \alpha, \ \ ||u||_p\geq \beta.
%\end{equation}
%The existence up to translation of a weak limit different from the null function derives from \eqref{allalieb} together with Lemma 2.1 of \cite{FLL} and with Lemma 6 %of \cite{Lieb}.
\end{proof}

\begin{lem}\label{pohoz}
Let $p \in (\frac{10}{3},6)$, $\lambda \in \R$. If $v\in H^1(\R^3)$  is a weak solution of
\begin{equation}\label{pde2}
-\Delta v + \left( |x|^{-1}*|v|^2\right)v -|v|^{p-2}v=\lambda v
\end{equation}
then $Q(v) =0$. Moreover if $\lambda \geq 0$, there exists a constant $c_0>0$ independent on $\lambda \in \R$ such that the only solution of \eqref{pde2} fulfilling $||v||_2^2\leq c_0$ is the null function.
\end{lem}
\begin{proof}
The following Pohozaev type identity holds for $v\in H^1(\R^3)$ weak solution of \eqref{pde2}, see \cite{DM},
\begin{equation*}
\frac{1}{2}\int_{\mathbb{R}^3} |\nabla v|^2dx  +\frac{5}{4}\int_{\mathbb{R}^3}\int_{\mathbb{R}^3}\frac{\left | v(x) \right |^2\left | v(y) \right |^2}{\left | x-y \right |}dxdy-\frac{3}{p}\int_{\mathbb{R}^3}\left | v\right |^pdx=\frac{3 \lambda}{2} \int_{\mathbb{R}^3}\left | v\right |^2dx.
\end{equation*}
By multiplying \eqref{pde2} by $v$ and integrating we derive a second identity
\begin{equation*}
\int_{\mathbb{R}^3} |\nabla v|^2dx  +\int_{\mathbb{R}^3}\int_{\mathbb{R}^3}\frac{\left | v(x) \right |^2\left | v(y) \right |^2}{\left | x-y \right |}dxdy-\int_{\mathbb{R}^3}\left | v\right |^pdx=\lambda \int_{\mathbb{R}^3}\left | v\right |^2dx.
\end{equation*}
With simple calculus we obtain the following relations
\begin{equation}\label{casolambda1}
\begin{split}
&A(v)+\frac{1}{4}B(v)+3\left( \frac{p-2}{2p}\right) C(v)=0,\\
&  (\frac{p-6}{3p-6})A(v)+(\frac{5p-12}{3p-6})\frac{B(v)}{2}=\lambda D(v).
\end{split}
\end{equation}
The first relation of \eqref{casolambda1} is $Q(v)=0$.  This identity together with the Gagliardo-Nirenberg inequality assures the existence of a constant $C(p)$ such that
\begin{equation}
A(v)-C(p) A(v)^{\frac{3(p-2)}{4}}D(v)^{\frac{6-p}{4}}  \leq A(v)+3\left( \frac{p-2}{2p}\right) C(v) \leq 0,
\end{equation}
i.e
\begin{equation}\label{crr}
A(v)^{\frac{10-3p}{4}}\leq C(p)D(v)^{\frac{6-p}{4}}.
\end{equation}
Now we recall that by the Hardy-Littlehood-Sobolev inequality and the
Gagliardo-Nirenberg inequality we have
\begin{equation}\label{impo}
B(v)\leq C A(v)^{\frac{1}{2}}D(v)^{\frac{3}{2}},
\end{equation}
then, from the second relation of \eqref{casolambda1}  we obtain
\begin{equation}\label{crrr}
\lambda D(v) \leq (\frac{p-6}{3p-6})A(v)+\tilde{C}(p)A(v)^{\frac{1}{2}}D(v)^{\frac{3}{2}}.
\end{equation}
Notice that \eqref{crr} tells us that, for any solution $u$ of \eqref{pde2} with small $L^2$-norm, $A(u)$ must be large. This fact assures that the left hand side of \eqref{crrr} cannot be non negative when $D(v)$ is sufficiently small.
\end{proof}
% Now, in order the show the compactness of the PS sequences we prove the following crucial lemma.
% \begin{lem}[Subadditivity]\label{subadditivity}
% Let $(0, \infty ) \ni c\rightarrow \gamma(c)$ be a continuous function  such that $\lim_{c\rightarrow 0}\frac{\gamma(c)}{c}=\infty$. Then for any $c>0$ there exists $c_0\in (0,c)$ such that
% \begin{equation}\label{sub}
% \gamma(c_0) < \gamma(c_1)+\gamma(c_0-c_1) \text{ for all } \ 0<c_1<c_0.
% \end{equation}
% \end{lem}
% \begin{proof}
% % Now  we argue as in \cite{L} observing that  the subadditivity inequality can be proved by showing that $\gamma(\theta c)< \theta \gamma(c)$ with $\theta>1$, indeed  (let us consider for simplicity $c_1>c-c_1$)
% % \begin{eqnarray}
% % \gamma(c)=\gamma(\frac{c}{c_1}c_1)<\frac{c}{c_1}\gamma(c_1)=\frac{c-c_1+c_1}{c_1}\gamma(c_1)=\\\frac{c-c_1}{c_1}\gamma(\frac{c_1}{c-c_1}(c-c_1))+\gamma(c_1)<\gamma(c_1)+\gamma(c-c_1).\nonumber
% % \end{eqnarray}
% We argue as in \cite{BS} by defining for $c>0$
% $$
% c_0:=min \{s\in [0,c] \text{ s.t. } \frac{\gamma(s)}{s}=\frac{\gamma(c)}{c}\}.
% $$
% By hypotesis $c_0>0$ and by continuity the function $\frac{\gamma(c)}{c}$ in the interval $(0, c_0)$ achieves the minimum only on $c=c_0$. The fact that $\frac{\gamma(c_0)}{c_0}<\frac{\gamma(c_1)}{c_1}$ and  $\frac{\gamma(c_0)}{c_0}<\frac{\gamma(c-c_1)}{c-c_1}$ implies the subadditivity relation \eqref{sub}.
% \end{proof}
\begin{lem}\label{main-abs1}
Let $p \in (\frac{10}{3}, 6)$. Assume that the bounded Palais-Smale sequence $\{u_{n}\} \subset S(c)$ given by Lemma  \ref{apriori} is weakly convergent, up to translations, to the nonzero function $\bar u.$ Moreover assume that
\begin{equation}\label{crucialcond}
 \forall c_1 \in (0,c), \ \ \gamma(c_1) > \gamma(c).
\end{equation}
% \begin{equation}\label{crucialcond2}
% Q(\bar u)=0
% \end{equation}
%
%\begin{equation}\label{(2)}
%T(u_n-\bar u) + T(\bar u)=T(u_n)+ o(1);
%\end{equation}
%\begin{equation}\label{(3)}
%T(\alpha_{n}(u_{n}-\bar u))-T( u_n-\bar u)=o(1) \ \text{where}\ \alpha_{n}=\frac{\rho^2-\mu_{0}^2}{\| u_n-\bar u\|_2}.
%\end{equation}
Then
 $\|u_n-\bar u\|\rightarrow 0$.
 In particular it follows that $\bar u\in S(c)$ and $F(\bar u)=\gamma(c).$
%Moreover if, as $n,m \to +\infty$
%\begin{equation}\label{(5)}
%<T'(u_{n}) - T'(u_{m}), u_{n} - u_{m}> = o(1) \ \ \text{ and } \ \
% <T'(u_{n}), u_{n}> = O(1)
%\end{equation}
%\begin{equation}\label{(6)}
%<T'(u_{n}), u_{n}> = O(1)
%\end{equation}
%then $\|u_n-\bar u\|_{H^m(\R^{N})}\rightarrow 0$.
\end{lem}

\begin{proof}
Let $T(u):=\frac{1}{4}B(u)+\frac{1}{p}C(u)$ such that
\begin{equation}\label{functional}
F(u):= \frac 12   \|\nabla u \|_2^2 +T(u).
\end{equation}
In \cite{BS1} or \cite{ZZ} it is shown that the nonlinear term $T$ fulfills the following splitting properties of Brezis-Lieb type (see \cite{BRLI}),
\begin{equation}\label{i}
T(u_n-\bar u) + T(\bar u)=T(u_n)+ o(1).
\end{equation}
We argue by contradiction and assume that $c_1=||\bar u||_2^2 < c$. Since $u_{n}-\bar u\rightharpoonup0$,
$$
\| u_n-\bar u\|_2^2+\|\bar u\|_2^2=\| u_n\|_2^2+o(1).
$$
%\begin{equation} \label{convergenza-norme}
%\alpha_{n} := \frac{\sqrt{\rho^2-\mu^2}}{\| u_n-\bar u\|_2} \to 1
%\end{equation}
Since $\{u_{n}\} \subset H^1(\R^3)$ is a bounded PS sequence at the mountain pass level, we get
$$
\frac 12 \| \nabla u_{n} \|_2^2 +T(u_n)=\gamma(c)+o(1)
$$
and by \eqref{i}, we deduce also
$$
\frac 12 \| \nabla (u_n-\bar u) \|_2^2 +\frac 12 \|\nabla \bar u\|_2^2+T(u_n-\bar u)+T(\bar u)=\gamma(c)+o(1).
$$
% By continuity we obtain
% $$
% \frac 12 \| \alpha_{n}(u_n-\bar u) \|_{D^{1,2}}^2 +\frac 12 \|\bar u\|_{D^{1,2}}^2 +T(\alpha_{n}(u_n-\bar u))+T(\bar u)=\gamma(c)+o(1).
% $$
Thanks to Lemmas \ref{prop1.3} and \ref{pohoz}, $\bar u \in V(c_1)$ and by Lemma \ref{mpestimate} we get
\begin{equation}\label{quasi}
F(u_n-\bar u)+ \gamma(c_1) \leq \gamma(c)+o(1).
\end{equation}
On the other hand,
\begin{equation}\label{sisi}
F(u_n-\bar u)-\frac{2}{3(p-2)}Q(u_n-\bar u)=\frac{3p-10}{6(p-2)}A(u_n-\bar u)+\frac{3p-8}{12(p-2)}B(u_n-\bar u)
\end{equation}
and
\begin{equation}\label{sisi3}
Q(u_n-\bar u)=Q(u_n-\bar u)+Q(\bar u)=Q(u_n)+o(1)=o(1).
\end{equation}

>From \eqref{sisi} and \eqref{sisi2} we deduce that $F(u_n-\bar u) \geq o(1)$. But then from (\ref{quasi}) we obtain a contradiction with \eqref{crucialcond}. This contradiction proves that $\| \bar u \|_2^2= c$ and $F(\bar u)\geq \gamma(c)$. Now still by  \eqref{quasi} we get $F(u_n -\bar u)\leq o(1)$ and thanks to \eqref{sisi} and \eqref{sisi3} $A(u_n -\bar u)=o(1)$. i.e $\| \nabla (u_{n} - \bar
u )\|_2 = o(1)$.
\end{proof}
\begin{lem}\label{main-abs2}
Let $p \in (\frac{10}{3},6)$. Assume that the bounded Palais-Smale sequence $\{u_{n}\} \subset S(c)$ given by Lemma  \ref{apriori} is weakly convergent, up to translations, to the nonzero function $\bar u.$  Moreover assume that
\begin{equation}\label{crucialcond2}
 \forall c_1 \in (0,c), \ \gamma(c_1) \geq \gamma(c)
\end{equation}
and that the Lagrange multiplier given by Proposition \ref{prop1.3} fulfills
$$\lambda_c \neq 0. $$
Then
 $\|u_n-\bar u\|\rightarrow 0$.
In particular it follows that $\bar u\in S(c)$ and $F(\bar u)=\gamma(c).$
\end{lem}
\begin{proof}
Let us argue as in Lemma \ref{main-abs1}. We obtain again
\begin{equation*}
F((u_n-\bar u))+ \gamma(c_1) \leq \gamma(c)+o(1),
\end{equation*}
\begin{equation*}
F(u_n-\bar u)-\frac{2}{3(p-2)}Q(u_n-\bar u)=\frac{3p-10}{6(p-2)}A(u_n-\bar u)+\frac{3p-8}{12(p-2)}B(u_n-\bar u)
\end{equation*}
and
\begin{equation*}
Q(u_n-\bar u)=Q(u_n-\bar u)+Q(\bar u)=Q(u_n)+o(1)=o(1).
\end{equation*}
Thanks to \eqref{crucialcond2} we conclude that
$$\frac{3p-10}{6(p-2)}A(u_n-\bar u)+\frac{3p-8}{12(p-2)}B(u_n-\bar u)=o(1).$$
Then
\begin{equation}\label{stimacruc}
A(u_n-\bar u)=o(1), B(u_n-\bar u)=o(1)\ \mbox{and also}\ C(u_n-\bar u)=o(1),
\end{equation}
since $Q(u_n-\bar u)=o(1)$. Now we use (5) of Proposition \ref{prop1.3}, i.e
$$A(u_n)-\lambda_cD(u_n)+B(u_n)+C(u_n)=A(\bar u)-\lambda_cD(\bar u)+B(\bar u)+C(\bar u)+o(1).$$
Thanks to the splitting properties of $A(u),B(u), C(u)$ and to \eqref{stimacruc}
we get
$$-\lambda_cD(u_n)=-\lambda_cD(\bar u)+o(1),$$
which implies $D(u_n-\bar u)=o(1)$, i.e $||u_n-\bar u||_2=o(1)
$. From this point we conclude as in the proof of Lemma \ref{main-abs1}.
\end{proof}

Admitting for the moment that $c \to \gamma(c)$ is non-increasing
(we shall prove it in the next section) we can now complete the proof of Theorem \ref{mainvero}.
\begin{proof}[Proof of Theorem \ref{mainvero}]
By Lemmas \ref{apriori} and \ref{compact}
there exists a bounded Palais-Smale sequence
$\{u_n\} \subset S(c)$ such that, up to translation,
$u_n \rightharpoonup u_c\neq 0.$ Thus, by Proposition \ref{prop1.3} there exists a $\lambda_c \in \R$ such that $(u_c,\lambda_c) \in H^1(\R^3)\backslash \{0\} \times \R$ solves (\ref{eq}). Now by Lemma \ref{pohoz} there exists a $c_0 >0$ such that $\lambda_c <0$ if $c \in (0,c_0)$. Also we know from Theorem \ref{maingamma} (ii) that
(\ref{crucialcond2}) holds. At this point the proof follows from Lemma \ref{main-abs2}.
\end{proof}
%Moreover such PS sequence fulfills $Q(u_n)=o(1).$
%The strong convergence in $H^1(\R^3)$ follows from Proposition \ref{prop1} and .\\

\section{The behaviour of  $c \to \gamma(c)$} \label{Section5}

In this section we give the proof of Theorem \ref{maingamma}. Let us denote
\begin{eqnarray}\label{1}
\gamma_1(c)=\inf_{u\in S(c)} \max_{t>0}F(u^t),
\end{eqnarray}
and
\begin{eqnarray}\label{2}
\gamma_2(c)=\inf_{u\in V(c)}F(u).
\end{eqnarray}

\begin{lem}\label{lm1} For $p \in (\frac{10}{3},6)$, we have:
$$\gamma(c)=\gamma_1(c)=\gamma_2(c).$$
\end{lem}
\begin{proof} When $p \in (\frac{10}{3},6)$, from  Lemma \ref{mpestimate}, we know that $\gamma(c)=\gamma_2(c)$.
In addition, by  Lemma \ref{growth}, it is clear that for any $u \in S(c)$, there exists a unique $t_0>0$, such that $u^{t_0} \in V(c)$ and $\max_{t>0}F(u^t)=F(u^{t_0})\geq \gamma_2(c)$, thus we get $\gamma_1(c)\geq \gamma_2(c)$. Meanwhile, for any $u \in V(c)$, $\max_{t>0}F(u_t)=F(u)$ and this readily implies that $\gamma_1(c) \leq \gamma_2(c)$. Thus we conclude that $\gamma_1(c)= \gamma_2(c)$.
\end{proof}

\begin{lem} \label{lemma2} We denote
$$f(a,b,c)=\max_{t>0} \left \{ a\cdot t^2+b\cdot t - c\cdot t^{\frac{3}{2}(p-2)} \right \},$$
where $p \in (\frac{10}{3},6)$ and $a>0, b\geq 0, c>0$ which are totally independent of $t$. Then the function: $(a,b,c) \longmapsto f(a,b,c)$ is continuous in $\mathbb{R}^+\times \mathbb{R}_-^c\times \mathbb{R}^+$ (here we denote $\mathbb{R}_-^c$ the non negative real number set).
\end{lem}

\begin{proof} Let
$g(a,b,c,t)= a\cdot t^2+b\cdot t - c\cdot t^{\frac{3}{2}(p-2)}$, then
$$\partial_t g(a,b,c,t)=2a\cdot t+b - \frac{3}{2}(p-2)\cdot c\cdot t^{\frac{3p-8}{2}},$$
$$\partial_{tt}^2 g(a,b,c,t)=2a - \frac{3p-6}{2}\cdot \frac{3p-8}{2}\cdot c\cdot t^{\frac{3p-10}{2}}.$$
It's not difficult to see that for any $(a_0,b_0,c_0)$ with $a_0>0, b_0\geq 0, c_0>0$, there exists a unique $t_1>0$, such that $\partial_t g(a_0,b_0,c_0,t_1)=0$ and $\partial_{tt}^2 g(a_0,b_0,c_0,t_1)<0$, thus $f(a_0,b_0,c_0)=g(a_0,b_0,c_0,t_1)$. Then applying the Implicit Function Theorem to the function $\partial_t g(a,b,c,t)$, we deduce the existence of a continuous function $t=t(a,b,c)$ in some neighborhood $O$ of $(a_0,b_0,c_0)$ that satisfies $\partial_t g(a,b,c,t(a,b,c))=0$, $\partial_{tt}^2 g(a,b,c,t(a,b,c))<0$. Thus $f(a,b,c)=g(a,b,c,t(a,b,c))$ in $O$. Now since the function $g(a,b,c,t)$ is continuous in $(a,b,c,t)$, it follows that $f(a,b,c)$ is continuous in $(a_0,b_0,c_0)$. The point $(a_0,b_0,c_0)$ being arbitrary this concludes the proof.
\end{proof}

\begin{lem}\label{prop1}
When $p \in (\frac{10}{3},6)$, the function $c \mapsto \gamma(c)$ is non increasing for $c>0$.
\end{lem}
\begin{proof}
To show that  $c\mapsto \gamma(c)$ is non increasing, it is enough to verify that: for any $c_1<c_2$ and $\varepsilon >0$ arbitrary, we have
\begin{eqnarray}\label{33}
\gamma(c_2)\leq \gamma(c_1)+\varepsilon.
\end{eqnarray}
By definition of $\gamma_2(c_1)$, there exists $u_1 \in V(c_1)$ such that $F(u_1)\leq  \gamma_2(c_1)+ \frac{\varepsilon}{2}$. Thus by Lemma \ref{lm1}, we have
\begin{eqnarray}\label{44}
F(u_1)\leq \gamma(c_1)+\frac{\varepsilon}{2}
\end{eqnarray}
and also
\begin{equation}\label{444}
F(u_1)=\max_{t>0}F(u_1^t).
\end{equation}
We truncate $u_1$ into a function with compact support $\widetilde{u}_1$ as follows. Let $\eta \in C_0^{\infty}(\mathbb{R}^3)$ be radial and such that
$$\eta (x)=\left \{ \begin{matrix}
1, \qquad  \qquad \left | x \right |\leq 1,\\
\in [0,1], \ 1<\left | x \right |<2, \\
0, \qquad \qquad \left | x \right |\geq 2.
\end{matrix} \right.$$
For any small $\delta>0$, let
\begin{eqnarray}\label{55}
\widetilde{u}_1(x)=\eta(\delta x)\cdot u_1(x).
\end{eqnarray}
It is standard to show that
$\widetilde{u}_1(x) \rightarrow u_1(x)$  in $ H^1(\mathbb{R}^3)$ as $ \delta \to 0.$ Then, by continuity, we have, as $\delta \to 0$,
\begin{equation}\label{avv}
A(\widetilde{u}_1)  \to A(u_1),\,
B(\widetilde{u}_1) \to B(u_1) \, \mbox{ and }
C(\widetilde{u}_1)  \to C(u_1).
\end{equation}
At this point applying Lemma \ref{lemma2}, we deduce that there exists $\delta >0$ small enough, such that
\begin{eqnarray}\label{77}
\max_{t>0}F(\widetilde{u}_1^t)&=& \max_{t>0} \Big\{ \frac{t^2}{2}A(\widetilde{u}_1)+ t B(\widetilde{u}_1) + t^{\frac{3}{2}(p-2)}C(\widetilde{u}_1)\Big\} \nonumber \\
&\leq & \max_{t>0} \Big\{ \frac{t^2}{2}A(u_1)+ t B(u_1) + t^{\frac{3}{2}(p-2)}C(u_1)\Big\}   + \frac{\varepsilon}{4} \nonumber \\
&= & \max_{t>0} F(u_1^t)   + \frac{\varepsilon}{4}.
\end{eqnarray}
Now let $v(x)\in C_0^{\infty}(\mathbb{R}^3)$ be radial and such that $supp \ v \subset B_{2R_{\delta}+1}\backslash B_{2R_{\delta}}$. Here $supp \ v$ denotes the support of $v$ and $R_{\delta}=\frac{2}{\delta}$. Then we define
$$ v_0=(c_2- \|\widetilde{u}_1\|_2^2)/\|v\|_2^2 \cdot v$$
for which we have $\|v_0\|_2^2=c_2- \|\widetilde{u}_1\|_2^2$. Finally letting $v_0^{\lambda}=\lambda^{\frac{3}{2}}v_0(\lambda x)$, for $\lambda \in (0,1)$, we have $\|v_0^{\lambda}\|_2^2=\|v_0\|_2^2$ and
\begin{equation}\label{7}
A( v_0^{\lambda} )= \lambda^2\cdot A( v_0 ), \,
B(v_0^{\lambda})  = \lambda\cdot B(v_0) \mbox{ and }
C(v_0^{\lambda} ) = \lambda^{\frac{3}{2}(p-2)}\cdot C(v_0 ).
\end{equation}
Now for any $\lambda \in (0,1)$ we define $w_{\lambda}=\widetilde{u}_1 + v_0^{\lambda}$. We observe that
\begin{equation}\label{suport}
dist\{supp\ \widetilde{u}_1, supp\ v_0^{\lambda} \} \geq \frac{2R_{\delta}}{\lambda}-R_{\delta}=\frac{2}{\delta}(\frac{2}{\lambda}-1).
\end{equation}
Thus $\|w_{\lambda}\|_2^2=\|\widetilde{u}_1\|_2^2+\|v_0^{\lambda}\|_2^2$ and $w_{\lambda} \in S(c_2)$. Also
\begin{equation}\label{11}
A( w_{\lambda})=A(\widetilde{u}_1)+A(v_0^{\lambda})\ \mbox{and } C(w_{\lambda})=C(\widetilde{u}_1) + C(v_0^{\lambda}).
\end{equation}
We claim that, for any $\lambda \in (0,1)$ ,
\begin{equation}\label{12}
\Big| B(w_{\lambda}) -  B(\widetilde{u}_1)- B(v_0^{\lambda})\Big| \leq \lambda \, \|\widetilde{u}_1\|_2^2\cdot \|  v_0^{\lambda}\|_2^2.
\end{equation}
Indeed, from (\ref{suport}),
$$\left ( \widetilde{u}_1+v_0^{\lambda } \right )^2(x)=\widetilde{u}_1^2(x)+\left ( v_0^{\lambda } \right )^2(x),\\
\left ( \widetilde{u}_1+v_0^{\lambda } \right )^2(y)=\widetilde{u}_1^2(y)+\left ( v_0^{\lambda } \right )^2(y).
$$
Thus
\begin{eqnarray*}
B(w_{\lambda})&=&\int_{\mathbb{R}^3} \int_{\mathbb{R}^3}\frac{\left ( \widetilde{u}_1+v_0^{\lambda } \right )^2(x)\cdot \left ( \widetilde{u}_1+v_0^{\lambda } \right )^2(y)}{|x-y|}dxdy \\
&=&\int_{\mathbb{R}^3} \int_{\mathbb{R}^3}\frac{ \widetilde{u}_1^2(x)\cdot \widetilde{u}_1^2(y)}{|x-y|}dxdy +2 \int_{\mathbb{R}^3} \int_{\mathbb{R}^3}\frac{ \widetilde{u}_1^2(x)\cdot \left ( v_0^{\lambda } \right )^2(y)}{|x-y|}dxdy \\
&+&\int_{\mathbb{R}^3} \int_{\mathbb{R}^3}\frac{ \left ( v_0^{\lambda } \right )^2(x)\cdot \left ( v_0^{\lambda } \right )^2(y)}{|x-y|}dxdy \\
&=&B(\widetilde{u}_1)+B(v_0^{\lambda })+2 \int_{\mathbb{R}^3} \int_{\mathbb{R}^3}\frac{ \widetilde{u}_1^2(x)\cdot \left ( v_0^{\lambda } \right )^2(y)}{|x-y|}dxdy
\end{eqnarray*}
with
\begin{eqnarray*}
\int_{\mathbb{R}^3} \int_{\mathbb{R}^3}\frac{ \widetilde{u}_1^2(x)\cdot \left ( v_0^{\lambda } \right )^2(y)}{|x-y|}dxdy
&=&\int_{supp\ \widetilde{u}_1 } \int_{supp\ v_0^{\lambda }}\frac{ \widetilde{u}_1^2(x)\cdot \left ( v_0^{\lambda } \right )^2(y)}{|x-y|}dxdy \\
&\leq& \frac{\delta \lambda }{2(2-\lambda )} \int_{supp\ \widetilde{u}_1 }\int_{supp\ v_0^{\lambda }}\widetilde{u}_1^2(x)\cdot \left ( v_0^{\lambda } \right )^2(y)dxdy \\
&\leq&\frac{\delta \lambda }{2(2-\lambda )}\left \| \widetilde{u}_1 \right \|_2^2\cdot \left \|v_0^{\lambda }  \right \|_2^2 \\
&\leq& \frac{\lambda }{2}\left \| \widetilde{u}_1 \right \|_2^2\cdot \left \|v_0^{\lambda }  \right \|_2^2
\end{eqnarray*}
and then (\ref{12}) holds. Now from (\ref{11}), (\ref{12}) and using (\ref{7}) we see that
\begin{equation}\label{14}
A(w_{\lambda}) \to A(\widetilde{u}_1 ), \, B(w_{\lambda}) \to B(\widetilde{u}_1 ) \mbox{ and } C(w_{\lambda}) \to C(\widetilde{u}_1 ), \ \mbox{as }\lambda \to 0.
\end{equation}
Thus from Lemma \ref{lemma2} we have that, fixing $\lambda >0$ small enough,
\begin{equation}\label{15}
\max_{t>0}F(w_{\lambda}^t) \leq \max_{t>0} F(\widetilde{u}_1^t)+\frac{\varepsilon}{4}.
\end{equation}
Now, using Lemma \ref{lm1}, (\ref{15}), (\ref{77}), \eqref{444} and \eqref{44} we have that
\begin{eqnarray*}
\gamma(c_2)\leq \max_{t>0}F(w_{\lambda}^t)&\leq& \max_{t>0} F(\widetilde{u}_1^t)+\frac{\varepsilon}{4} \\
&\leq &\max_{t >0} F(u_1^t) + \frac{\varepsilon}{2} \\
&=& F(u_1) + \frac{\varepsilon}{2}  \leq \gamma(c_1)+\varepsilon
\end{eqnarray*}
and this ends the proof.
\end{proof}

\begin{lem}\label{continuity}
When $p \in (\frac{10}{3},6)$, $c \mapsto \gamma(c)$ is continuous at each $c>0$.
\end{lem}
\begin{proof}
Since, by Lemma \ref{prop1}, $c \to \gamma(c)$ is non increasing proving that it is continuous at $c>0$ is equivalent to show that for any sequence $c_n \to c^+$
\begin{equation}\label{c1}
\gamma(c) \leq \lim_{c_n \to c^+}\gamma(c_n).
\end{equation}
Let $\varepsilon >0$ be arbitrary but fixed. By Lemma \ref{mpestimate} we know that there exists  $u_n \in V(c_n)$ such that
\begin{equation}\label{c2}
F(u_n) \leq \gamma(c_n) + \frac{\varepsilon }{2}.
\end{equation}
We define $\widetilde{u}_n=\frac{c}{c_n}\cdot u_n := \rho _n \cdot u_n$. Then $\widetilde{u}_n \in S(c)$  and  $\rho _n  \to 1^-$. In addition
\begin{eqnarray}\label{add}
\gamma(c)&\leq& \max_{t>0}F(\widetilde{u}_n^t) \nonumber \\
&=&\max_{t>0}\{\frac{t^2}{2}\rho _n^2 A(u_n)+\frac{t}{4}\rho _n^4B(u_n)+\frac{t^{\frac{3p-6}{2}}}{p}\rho_n^p C(u_n) \}.
\end{eqnarray}
Since $u_n \in V(c_n)$ and $c_n \to c^+$, using the identity
\begin{eqnarray}\label{identity}
F(u_n)-\frac{2}{3(p-2)}Q(u_n)=\frac{3p-10}{6(p-2)}A(u_n)+\frac{3p-8}{12(p-2)}B(u_n),
\end{eqnarray}
it is not difficult to check that $A(u_n), B(u_n)$ and $C(u_n)$ are bounded both from above and from zero. Thus without restriction we can get that
\begin{equation*}
A(u_n) \to A >0, \, B(u_n) \to B \geq 0 \quad \mbox{and} \quad C(u_n) \to C<0.
\end{equation*}
Indeed, $A\geq 0, B\geq 0, C \leq 0$ are trivial and it is also easy to verify by contradiction that $A \neq 0, C \neq 0$ from \eqref{impo}, (\ref{identity}) and the fact
$$Q(u_n)=A(u_n)+\frac{1}{4}B(u_n)+\frac{3p-6}{2p}C(u_n)=0.$$
Now recording that $\rho_n \to 1^-$, using Lemma \ref{lemma2} twice, we get from (\ref{add}), for any $n \in \N$ sufficiently large
\begin{eqnarray}\label{c5}
\max_{t>0}F(\widetilde{u}_n^t)&\leq &\max_{t>0}\{ (\frac{A}{2})t^2 +  (\frac{B}{4}) t - (- \frac{C}{p}) t^{\frac{3}{2}(p-2)}\} + \frac{\varepsilon}{4}\nonumber \\
&\leq&\max_{t>0} \{ (\frac{A(u_n)}{2}) t^2 + (\frac{B(u_n)}{4}) t- (- \frac{C(u_n)}{p}) t^{\frac{3}{2}(p-2)}\} + \frac{\varepsilon}{2}\nonumber \\
& = & \max_{t>0}F(u_n^t) + \frac{\varepsilon}{2}
= F(u_n)+ \frac{\varepsilon }{2}.
\end{eqnarray}
Now from (\ref{c2}) and (\ref{c5}) it follows that
$
\gamma(c) \leq \gamma(c_n)+ \varepsilon$
for $n \in \N$ large enough and since $\varepsilon >0$ is arbitrary (\ref{c1}) holds.
\end{proof}

\begin{lem}\label{strictdecrease} Let  $p \in (\frac{10}{3},6)$ and $(u_c, \lambda_c) \in H^1(\mathbb{R}^3) \times \mathbb{R}^*$ solves
\begin{equation*}
-\Delta v - \lambda v + (\left | x \right |^{-1}\ast \left | v \right |^2) v-|v|^{p-2}v=0  \mbox{ in }  \mathbb{R}^{3},
\end{equation*}
with $F(u_c)=\inf_{u \in V(c)}F(u)=\gamma(c)$. Then if $\lambda_c <0$ ($\lambda_c >0)$ the function  $c \to \gamma(c)$ is strictly decreasing (increasing) in a neighborhood of $c$.
\end{lem}
\begin{proof}
The proof follows as a consequence of the implicit function theorem.\\
Let us consider the following rescaled functions $u_{t, \theta}(x)=\theta^{\frac{3}{2}}t^{\frac{1}{2}}u_c(\theta x)\in S(tc)$ with $\theta\in (0, \infty)$ and $t \in (0,\infty)$.
We define the following quantities
\begin{equation}\label{alpha}
\alpha(t, \theta)= F(u_{t, \theta}),
\end{equation}
\begin{equation}\label{alpha}
\beta(t, \theta)= Q(u_{t, \theta}).
\end{equation}
Simple calculus shows that
\begin{equation}
\frac{\partial \alpha(t,\theta)}{\partial t}_{|_{(1,1)}}=\frac{1}{2}\left ( A(u_c)+B(u_c)+C(u_c) \right ) =\frac{1}{2} \lambda_c c
\end{equation}
\begin{equation}
\frac{\partial \alpha(t,\theta)}{\partial \theta}_{|_{(1,1)}}=0, \quad \frac{\partial^2 \alpha(t,\theta)}{\partial^2 \theta}_{|_{(1,1)}}<0.
\end{equation}
Following the classical Lagrange Theorem we get, for any $\delta_t \in \R$, $ \delta_{\theta} \in \R$,
\begin{equation}
\alpha(1+\delta_t, 1+\delta_{\theta})=\alpha(1,1)+\delta_t \frac{\partial \alpha(t,\theta)}{\partial t}_{|_{(\bar t,\bar \theta)}}+\delta_{\theta} \frac{\partial \alpha(t,\theta)}{\partial \theta}_{|_{(\bar t,\bar \theta)}}
\end{equation}
where $|1- \bar t| \leq |\delta _t|$ and $|1- \bar \theta| \leq |\delta_{\theta}|$,
and by continuity, for sufficiently small $\delta_{t} >0$  and sufficiently small  $|\delta_{\theta}|,$
\begin{equation}\label{crucial}
\alpha(1+\delta_t, 1+\delta_{\theta})<\alpha(1,1) \quad \text{ if } \quad \lambda_c <0
\end{equation}
\begin{equation}\label{crucialbis}
\alpha(1-\delta_t, 1+\delta_{\theta})<\alpha(1,1) \quad \text{ if } \quad \lambda_c >0.
\end{equation}
To conclude the proof it is enough to show that $\beta(t,u)=0$ in a neighborhood of $(1,1)$ is the graph of a function $g:[1-\varepsilon, 1+\varepsilon]\rightarrow \R$ with $\varepsilon>0$, such that $\beta(t,g(t))=0$ for $t \in [1-\varepsilon, 1+\varepsilon]$. Indeed
in this case we have when $\lambda_c <0$  by (\ref{crucial})
$$\gamma((1+\varepsilon)c)=\inf_{u \in V((1+\varepsilon)c)}F(u)\leq F(u_{1+\varepsilon, g(1+\varepsilon)})<F(u_c)=\gamma(c)$$
and when $\lambda_c >0$ we have by (\ref{crucialbis})
$$\gamma((1-\varepsilon)c)=\inf_{u \in V((1-\varepsilon)c)}F(u)\leq F(u_{1-\varepsilon, g(1-\varepsilon)})<F(u_c)=\gamma(c).$$
To show the graph property by the Implicit Function Theorem it is sufficient to show that
\begin{equation}\label{implicit}
\frac{\partial \beta(t,\theta)}{\partial \theta}_{|_{(1,1)}}\neq 0.
\end{equation}
By simple calculus we get
$$\frac{\partial \beta(t,\theta)}{\partial \theta}_{|_{(1,1)}}=2A(u_c)+\frac{B(u_c)}{4}+\frac{1}{p}(\frac{3}{2}(p-2))^2C(u_c).$$
Using the fact that $Q(u_c)=0$ we then obtain
$$\frac{\partial \beta(t,\theta)}{\partial \theta}_{|_{(1,1)}} = (5 - \frac{3}{2}p) A(u_c) + (1 - \frac{3}{8}p) B(u_c).$$
Then, since $p > \frac{10}{3}$ we see that to have
$$ \frac{\partial \beta(t,\theta)}{\partial \theta}_{|_{(1,1)}} =0$$ necessarily  $A(u_c)=0$ and $B(u_c)=0$. Thus the derivative is never zero.
\end{proof}

\begin{lem}\label{limzero}
We have $\gamma(c) \to  \infty$ as $c \to 0$.
\end{lem}
\begin{proof}

By Theorem \ref{mainvero} we know that for any $c>0$ sufficiently small there exists a couple $(u_{c}, \lambda_c)\in H^1(\R^3)\times \R^-$ solution of \eqref{eq} with $||u_c||_2^2=c$ and $F(u_c)=\gamma(c)$. In addition by Lemma
\ref{pohoz}, $Q(u_c)=0$. Thus $u_c \in H^1(\R^3)$ fulfills \\
%\begin{equation}\label{21}
%\lambda_c \cdot \|u_c\|_2^2 = A( u_c)+B(u_c)+C(u_c),
%\end{equation}
\begin{equation}\label{22}
0 = Q(u_c) = A(u_c) + \frac{1}{4}B(u_c) + \frac{3(p-2)}{2p}C(u_c)
\end{equation}
\begin{equation}\label{23}
\gamma(c)=F(u_c) = \frac{1}{2} A(u_c) +\frac{1}{4}B(u_c)+\frac{1}{p}C(u_c).
\end{equation}
We deduce from (\ref{22}) that
$A(u_c) \leq - \frac{3(p-2)}{2p}C(u_c)$ and thus it follows from Gagliardo-Nirenberg inequality that
$$\|\nabla u_c\|_2^2 \leq \frac{3(p-2)}{2p} \|u_c\|_p^p \leq \widetilde{C}(p)\cdot \|\nabla u_c\|_2^{\frac{3(p-2)}{2}}\cdot \| u_c\|_2^{\frac{6-p}{2}},$$
i.e
\begin{eqnarray}\label{115}
1\leq \widetilde{C}(p)\cdot \|\nabla u_c\|_2^{\frac{3p-10}{2}}\cdot c^{\frac{6-p}{4}}.
\end{eqnarray}
Since $p \in (\frac{10}{3},6)$, we obtain that
\begin{eqnarray}\label{31}
\|\nabla u_c\|_2^2 \rightarrow \infty, \quad \mbox{as } \ c \to 0.
\end{eqnarray}
Now from (\ref{22}) and (\ref{23}) we deduce that
\begin{eqnarray}\label{113}
\gamma(c)=F(u_c) = \frac{3p-10}{6(p-2)}A(u_c)+\frac{3p-8}{12(p-2)}B(u_c).
\end{eqnarray}
and thus from (\ref{31}) we get immediately  that
$\gamma(c) \rightarrow \infty$  as $ c\to 0.$
\end{proof}

\section{Proof of Theorem \ref{naturalconstraint} and Lemma \ref{description}}\label{Section51}

In this section we prove Theorem \ref{naturalconstraint}. Let us first show

\begin{lem}\label{minimizer}
Let $p\in (\frac{10}{3}, 6)$, for each $u_c  \in \mathcal{M}_c $ there exists a $\lambda_c \in \mathbb{R}$ such that $(u_c, \lambda_c) \in H^1(\R^3) \times \R$ solves (\ref{eq}).
\end{lem}
\begin{proof} From Lagrange multiplier theory, to prove the lemma, it suffices to show that any $u_c  \in \mathcal{M}_c $ is a critical point of $F(u)$ constrained on $S(c)$. 

Let $u_c \in \mathcal{M}_c $ and assume, by contradiction, that $\|F'|_{S(c)}(u_c)\|_{H^{-1}(\R ^3)} \neq 0$.  Then, by the continuity of $F'$, there exist $\delta >0, \mu >0$ such that
$$ v \in B_{u_c}(3\delta)\ \Longrightarrow \ \|F'|_{S(c)}(v)\|_{H^{-1}(\R ^3)} \geq \mu,$$
where $B_{u_c}(\delta) : =\{v\in S(c)  \ :\ \|v-u_c\| \leq \delta \}$.

Let $\varepsilon:= \min\{\gamma(c)/4, \mu \delta /8 \}$. We claim that it is possible to construct a deformation on $S(c)$ such that
\begin{itemize}
  \item [(i)] $\eta(1,v)=v$ if $v \notin F^{-1}([\gamma(c)-2\varepsilon, \gamma(c)+2\varepsilon])$,
  \item [(ii)] $\eta(1,F^{\gamma(c)+\varepsilon}\bigcap B_{u_c}(\delta)) \subset F^{\gamma(c)-\varepsilon}$,
  \item [(iii)] $F(\eta(1,v)) \leq F(v)$, $\forall\ v \in S(c)$.
\end{itemize}
Here, $F^d:= \{u\in S(c) :  F(u)\leq d \}$. For this we use the pseudo gradient flow on $S(c)$ defined in (\ref{Flow}) but where now  $g: S(c) \rightarrow [0,\delta]$ satisfies
\begin{equation*}
g(v)=\left\{
\begin{array}{ll}
\delta\  \text{ if }  v \in B_{u_c}(2\delta) \bigcap F^{-1}([\gamma(c)-\varepsilon, \gamma(c)+\varepsilon]) \\
0\   \text{ if }  v \notin F^{-1}([\gamma(c)-2\varepsilon, \gamma(c)+2\varepsilon]). 
\end{array}
\right.
\end{equation*}
With this definition clearly (i) and (iii) hold. To prove (ii) first observe that  if $v \in F^{\gamma(c)+\varepsilon}\bigcap B_{u_c}(\delta)$, then $\eta(t,v) \in B_{u_c}(2\delta)$ for all $ t\in [0,1]$. Indeed
\begin{eqnarray*}
\left \| \eta (t,v)-v \right \|&=& \left \| \int_{0}^{t}-g(\eta (s,v))\frac{Y(\eta (s,v))}{\left \| Y(\eta (s,v)) \right \|}ds \right \| \\
&\leq& \int_{0}^{t} \left \| g(\eta (s,v)) \right \|ds \leq t \delta \leq \delta.
\end{eqnarray*}
In particular for $ s \in [0,1], g(\eta(s,v)) = \delta$ as long as $F(\eta(s,v)) \geq \gamma(c) - \varepsilon.$ Thus if we assume that there exists a $v \in F^{\gamma(c)+\varepsilon}\bigcap B_{u_c}(\delta)$ such that $F(\eta(1,v)) > \gamma(c) - \varepsilon$ we have
\begin{eqnarray*}
F(\eta (1,v))&=& F(v)+ \int_{0}^{1} \frac{d}{dt}F(\eta (t,v))dt \\
&=&  F(v)+ \int_{0}^{1} \langle dF(\eta (t,v)), -g(\eta (t,v)) \frac{Y(\eta (t,v))}{\left \| Y(\eta (t,v)) \right \|} \rangle dt  \\
&\leq & F(v)- \frac{\mu \delta}{4}\leq \gamma(c) + \varepsilon - \frac{\mu \delta }{4} \leq \gamma(c) - \varepsilon,
\end{eqnarray*}
i.e. $\eta (1,v) \in F^{\gamma(c)-\varepsilon} $. This contradiction proves that (ii) also hold.\\

Now let $g \in \Gamma_c$ be the path constructed in the proof of Lemma \ref{mpestimate} by choosing $v= u_c \in V(c)$. We claim that
\begin{eqnarray}\label{www}
\max_{t \in [0,1]}F(\eta(1,g(t)))<\gamma(c).
\end{eqnarray}
By (i) and Remark \ref{allok} we have $\eta(1,g(t)) \in \Gamma_c$. Thus if (\ref{www}) holds, it contradicts the definition of $\gamma(c)$. To prove (\ref{www}), we distinguish three cases:\\
a) If $g(t) \in S(c) \setminus B_{u_c}(\delta)$, then using (iii) and Lemma \ref{growth} (6),
$$F(\eta(1,g(t)))\leq F(g(t))< F(u_c)=\gamma(c).$$
b) If $g(t) \in F^{\gamma(c)-\varepsilon}$, then by (iii)
$$F(\eta(1,g(t)))\leq F(g(t))\leq \gamma(c)-\varepsilon.$$
c) If $g(t) \in F^{-1}([\gamma(c)-\varepsilon, \gamma(c)+ \varepsilon])\bigcap B_{u_c}(\delta)$, then by (ii)
$$F(\eta(1,g(t)))\leq \gamma(c)-\varepsilon.$$
Note that since $F(g(t)) \leq \gamma(c)$, for all $t \in [0,1]$ one of the three cases above must occurs. This proves that (\ref{www}) hold and the proof of the lemma is completed.
\end{proof}

\begin{proof} [Proof of Theorem \ref{naturalconstraint}]
We know from Lemma \ref{minimizer} that to each $u_c \in \mathcal{M}_c $ is associated a $\lambda_c \in \R$ such that $(u_c, \lambda_c) \in H^1(\R^3) \times \R$ is  solution of (\ref{eq}). Now using Lemmas \ref{prop1} and \ref{strictdecrease} we deduce that necessarily $\lambda_c \leq 0$.
\end{proof}

In view of Theorem \ref{naturalconstraint} it is reasonable to wonder if a solution of (\ref{eq}) may exists for a $\lambda > 0$. In that direction we have

\begin{lem}\label{lm1.6} We assume that $p \in (3,6)$ and $\lambda >0$, then the equation
$$-\Delta v-\lambda v+(\left | x \right |^{-1} \ast \left | v \right |^2) v-|v|^{p-2}v = 0\ \mbox{ in }\ \R^3$$
has only  trivial solution in $H_r^1(\mathbb{R}^3)$.
\end{lem}
To prove Lemma \ref{lm1.6} we use the following result due to Kato.
\begin{lem}(Kato \cite{TK})\label{lm1.7}
If we assume that the radial function $q(r)$ satisfies $q(r)=o(r^{-1})$ as $r \to \infty$ and that $m>0$ is a constant then the equation
\begin{eqnarray}\label{1.27}
-\Delta v+q(r)v= m v ,
\end{eqnarray}
has no nontrivial solution in $H_r^1(\mathbb{R}^3)$.
\end{lem}
\begin{proof}[Proof of the Lemma \ref{lm1.6}] We change the equation into the form of (\ref{1.27}):
\begin{eqnarray}\label{1.28}
-\Delta v+q(r)v=\lambda_c v,
\end{eqnarray}
where here $q(r)=|x|^{-1} * |v|^2-|v|^{p-2}$. To prove this lemma it suffices to verify the condition $q(r)=o(r^{-1})$ as $r \to \infty$. By the results of \cite{S} on the decay of radial functions there exists a $\widetilde{c}>0$, such that
$$|v(x)|\leq \widetilde{c} \cdot \|v\| \cdot \frac{1}{|x|}, \quad |x| \geq 1.$$
Thus since $p\geq 3$ it follows that
\begin{eqnarray*}
|v(x)|^{p-2}\leq \widetilde{c} \cdot \|v\|^{p-2} \cdot \frac{1}{|x|},\ \mbox{ when } \ |x|>0 \mbox{ is large.}
\end{eqnarray*}
Now for any radial function $f(x)=f(|x|)$ by the \emph{Newton's Theorem}, see \cite{LB},
$$\int_{\mathbb{R}^N}\frac{f(|y|)}{|x-y|^{N-2}}dy=\int_{\mathbb{R}^N}\frac{f(|y|)}{\max \{|x|,|y|\}}dy.$$
Thus we have
\begin{eqnarray*}
|x|^{-1}* |v|^2 &=&\int_{\mathbb{R}^3}\frac{|v(|y|)|^2}{|x-y|}dy=\int_{\mathbb{R}^3}\frac{|v(|y|)|^2}{\max \{|x|,|y|\}}dy\\
&\leq& \int_{\mathbb{R}^3}\frac{|v(|y|)|^2}{|x|}dy=\frac{\|v\|_2^2}{|x|}.
\end{eqnarray*}
Hence
\begin{eqnarray*}
|q(r)|&\leq& ||x|^{-1}* |v|^2|+|v|^{p-2}\\
&\leq&\frac{1}{|x|}(\|v\|_2^2+\widetilde{c}\cdot \|v\|^{p-2}),\ \mbox{ for } \ |x|>0 \mbox{ large},
\end{eqnarray*}
and $q(r)=o(r^{-1})$ as $r \to \infty$.
\end{proof}

\begin{proof}[Proof of Lemma \ref{description}]
Let $u_c \in H^1(\mathbb{R}^3, \mathbb{C})$ with $u_c\in V(c)$. Since $\|\nabla |u_c|\|_2\leq \|\nabla u_c \|_2$ we have that $F(|u_c|)\leq F(u_c)$ and $Q(|u_c|)\leq Q(u_c)=0$. In addition, by Lemma 2.2, there exists $t_0\in (0,1]$ such that $Q(|u_c|^{t_0})=0$. We claim that
\begin{equation}\label{below}
F(|u_c|^{t_0})\leq t_0 \cdot F(u_c).
\end{equation}
Indeed, due (\ref{sisi2}) and since $Q(|u_c|^{t_0})=Q(u_c)=0$, we have
\begin{eqnarray*}
F(|u_c|^{t_0}) &=& t_0^2\cdot \frac{3p-10}{6(p-2)}\|\nabla |u_c|\|_2^2 + t_0 \cdot \frac{3p-8}{12(p-2)}T(|u_c|) \\
&=& t_0\cdot \left ( t_0 \cdot \frac{3p-10}{6(p-2)}\|\nabla |u_c|\|_2^2 + \frac{3p-8}{12(p-2)}T(u_c)  \right ) \\
&\leq& t_0\cdot \left ( \frac{3p-10}{6(p-2)}\|\nabla u_c\|_2^2 + \frac{3p-8}{12(p-2)}T(u_c)  \right ) \\
&=& t_0 \cdot F(u_c).
\end{eqnarray*}
Thus if $u_c \in H^1(\mathbb{R}^3, \mathbb{C})$ is a minimizer of $F(u)$ on $V(c)$ we have
$$F(u_c)=\inf_{u \in V(c)}F(u)\leq F(|u_c|^{t_0})\leq t_0 \cdot F(u_c),$$
which implies $t_0=1$ since $t_0 \in (0,1]$. Then $Q(|u_c|)=0$ and we conclude that
\begin{eqnarray}
\|\nabla |u_c|\|_2 =  \|\nabla u_c \|_2\quad \mbox{and\quad } F(|u_c|)=F(u_c).
\end{eqnarray}
Thus point (i) follows. Now since $|u_c|$ is a minimizer of $F(u)$ on $V(c)$ we know by Theorem \ref{naturalconstraint} that it satisfies (\ref{eq}) for some $\lambda_c \leq 0$. By elliptic regularity theory and the maximum principle it follows that
$|u_c| \in C^1(\R^3, \R)$ and $|u_c|>0$. At this point, using that $\|\nabla |u_c|\|_2 =  \|\nabla u_c \|_2$ the rest of the proof of point (ii) is exactly the same as in the proof of Theorem 4.1 of \cite{HAST}.
\end{proof}

\section{Proof of Theorems \ref{maingamma} and  \ref{th3.1}  }\label{Section6}

 In \cite{IARU} the authors consider the functional $F(u)$ as a free functional defined in the real space
\begin{eqnarray*}
E := \{ u \in D^{1,2}(\R^3) : \int_{\R^3}\int_{\R^3} \frac{u^2(x)u^2(y)}{|x-y|}dxdy < \infty \}
\end{eqnarray*}
equipped with the norm
$$||u||_E := \Big( \int_{\R^3} |\nabla u(x)|^2 dx + \Big( \int_{\R^3}\int_{\R^3} \frac{u^2(x)u^2(y)}{|x-y|}dxdy\Big)^{\frac{1}{2}}\Big)^{\frac{1}{2}}.$$
Clearly $H^1(\R^3, \R) \subset E$. They show, see Theorem 1.1 and Proposition 3.4 in \cite{IARU}, that $F(u)$ has in $E$ a least energy solution whose energy is given by the mountain pass level
\begin{equation}\label{ruizlevel}
m := \inf_{\gamma \in \Gamma} \max_{t \in [0,1]}F(\gamma(t)) >0
\end{equation}
where
$$\Gamma:= \{ \gamma \in C([0,1], E), \gamma(0) =0, F(\gamma(1)) < 0\}.$$
\begin{lem}\label{boundbelow}
For any $c>0$ we have $\gamma(c) \geq m$ where $m >0$ is given in (\ref{ruizlevel}).
\end{lem}
\begin{proof}
We fix an arbitrary $c >0$. From Lemma \ref{description} we know
that the infimum of $F(u)$ on $V(c)$ is reached by real functions.
As a consequence in the definition of $\gamma(c)$, see in particular
(\ref{1}), we can restrict ourself to paths in $H^1(\R^N, \R)$ instead
of $H^1(\R^N, \C)$. To prove the lemma it suffices to show that for
any $g \in \Gamma_c$ there exists a $\gamma \in \Gamma$ such that
\begin{equation}\label{control}
\max_{t \in [0,1]}F(g(t)) \geq \max_{t \in [0,1]}F(\gamma(t)).
\end{equation}
Let $v\in S(c)$ be arbitrary but fixed. Letting  $v^{\theta}(x)=\theta^{\frac{3}{2}}v(\theta x)$ we have $v^{\theta}\in S(c)$ for any $\theta>0$. Also taking $\theta>0$ sufficiently small, $v^{\theta} \in A_{K_c}$. Now for $g \in \Gamma_c$ 
arbitrary but fixed, let $\gamma_{\theta}(t)\in
C([\frac{1}{4},\frac{1}{2}],A_{K_c})$ satisfies
$\gamma_{\theta}(\frac{1}{4})=v^{\theta}, \gamma_{\theta}(\frac{1}{2})=g(0)$, and consider $\gamma(t)$ given by
$$ \gamma(t) = \left \{
\begin{matrix}
\ 4tv^{\theta},\qquad \ 0 \leq t \leq \frac{1}{4}, \\
\ \gamma_{\theta}(t),\qquad  \frac{1}{4} \leq t \leq \frac{1}{2}, \\
\ g(2t-1), \quad \frac{1}{2} \leq t \leq 1.
\end{matrix}
\right.$$
Since $S(c)\subset H^1(\R^3) \subset E$ by construction $\gamma \in \Gamma$. Now direct calculations show that, taking $\theta>0$ small enough,  $F(4tv^{\theta})\leq F(v^{\theta})$ for any $t\in[0,\frac{1}{4}]$. Thus
$$\max_{t \in [0,1]}F(\gamma(t)) = \max_{t \in [\frac{1}{4},1]} F(\gamma(t)).$$
Recalling that $\gamma_{\theta}(t) \in A_{K_c}$ for any $t \in [ \frac{1}{4},
\frac{1}{2}]$, we conclude from Theorem \ref{thm1.2} that
$$\max_{t \in [0,1]}F(\gamma(t)) = \max_{t \in [\frac{1}{2},1]} F(\gamma(t))  = \max_{t \in [0,1]}F(g(t))$$
and (\ref{control}) holds. This proves the lemma.

%We fix an arbitrary $c >0$. From Lemma \ref{description} we know that the infimum of $F(u)$ on $V(c)$ is reached by real functions. As a consequence in the definition of $\gamma(c)$, see in particular (\ref{1}), we can restrict ourself to paths in $H^1(\R^N, \R)$ instead of $H^1(\R^N, \C)$. To prove the lemma it suffices to show that for any $g \in \Gamma_c$ there exists a $\gamma \in \Gamma$ such that
%\begin{equation}\label{control}
%\max_{t \in [0,1]}F(g(t)) \geq \max_{t \in [0,1]}F(\gamma(t)).
%\end{equation}
%Let $g \in \Gamma_c$ be arbitrary but fixed and consider $\gamma(t)$ given by
%$$ \gamma(t) = \left \{
%\begin{matrix}
%\ 2tg(0),\qquad \ 0 \leq t \leq \frac{1}{2}, \\
%\ g(2t-1), \quad \frac{1}{2} \leq t \leq 1.
%\end{matrix}
%\right.$$
%Since $ H^1(\R^3) \subset E$ by construction $\gamma \in \Gamma.$ Now since $2t g(0) \in A_{K_c}$ for any $t \in [0, \frac{1}{2}]$ we have
%$$\max_{t \in [0,1]}F(\gamma(t)) = \max_{t \in [\frac{1}{2},1]} F(\gamma(t)) = \max_{t \in [0,1]}F(g(t))$$
%and (\ref{control}) hold. This proves the lemma.
\end{proof}

\begin{lem}\label{liminf}
There exists $\gamma(\infty) >0$ such that $\gamma(c) \to \gamma(\infty)$ as $c \to \infty.$
\end{lem}
\begin{proof}
The existence of a limit follows directly from the fact that $c \to \gamma(c)$ is non-increasing. Now because of Lemma \ref{boundbelow} the limit is strictly positive.
\end{proof}

\begin{proof} [Proof of Theorem \ref{th3.1}] As we already mentioned this proof is largely due to L. Dupaigne. It also uses arguments from \cite{TC} and \cite{FO}. We  divide the proof into two steps. \\

\emph{Step 1}: Regularity and vanishing: let $(u, \lambda)$ with $u \in E$ and $\lambda \leq 0$ solves (\ref{eq}), then $u \in L^{\infty}(\R^{3})\bigcap C^1(\R^{3})$ and $u(x) \rightarrow 0$, as $|x| \to \infty$.\\

We set $\phi_{u}(x)= \frac{1}{4 \pi |x|}* u^2.$ Clearly since $u \in E$ then $\phi_{u} \in D^{1,2}(\R^3)$. We denote $H=-\Delta + ( 1 - \lambda )$. Since $\lambda \leq 0$, $H^{-1}$ exists in $L^{\eta}(\R^3)$ for all $\eta \in (1,\infty)$. The operators $H$ and $-\Delta$ being closed in $L^{\eta}(\R^3)$ with domain $D(H)\subset D(-\Delta)$, it follows from the Closed Graph Theorem that there exists a constant $\widetilde{C} >0$ such that
\begin{eqnarray}\label{va1}
\|\Delta u \|_{\eta} \leq \widetilde{C} \|H u\|_{\eta},
\end{eqnarray}
 for any  $u \in D(H)$. Now we write (\ref{eq}) as
\begin{eqnarray}\label{va3}
u=H^{-1}u-H^{-1}(\phi_{u} u)+ H^{-1}(|u|^{p-2}u)
\end{eqnarray}
and we claim that
\begin{equation}\label{keys}
H^{-1} u \in L^3 \cap L^{\infty}(\mathbb{R}^3) \mbox{ and } H^{-1}(\phi_{u} u) \in L^2 \cap L^{\infty}(\mathbb{R}^3).
\end{equation}
Indeed, $u\in L^q(\mathbb{R}^3)$ for all $q\in [3,6]$, see \cite{R2}, and from (\ref{va1}) and Sobolev's embedding theorem, we obtain
\begin{eqnarray}\label{va4}
H^{-1}u \in W^{2,q}(\mathbb{R}^3)\hookrightarrow L^{\infty}(\mathbb{R}^3),\ \forall\ q\in [3,6].
\end{eqnarray}
Now since $\phi_{u}\in \mathcal{D}^{1,2}(\R^3)\hookrightarrow L^6(\R^3)$, by H\"older inequality,  $\phi_{u}u\in L^t(\mathbb{R}^3)$ holds for any $t\in [2,3]$ and we have
\begin{eqnarray}\label{va5}
H^{-1}(\phi_{u}u) \in W^{2,t}(\mathbb{R}^3)\hookrightarrow L^{\infty}(\mathbb{R}^3),\ \forall\ t\in [2,3].
\end{eqnarray}
At this point the claim is proved.
Next we denote
\begin{eqnarray}\label{va6}
v:=u+H^{-1}(\phi_{u}u)-H^{-1}u.
\end{eqnarray}
By interpolation, and using (\ref{keys}), we see that $v\in L^q(\mathbb{R}^3)$ for all $q\in [3,6]$. Now since $u\in L^q(\R^3)$, for all $ q\in [3,6]$, (\ref{va3}) implies that
\begin{eqnarray}\label{va7}
H v=|u|^{p-2}u \in L^{\frac{q}{p-1}}(\mathbb{R}^3).
\end{eqnarray}
By (\ref{va1}) and Sobolev's embedding theorem, we conclude from (\ref{va7}) that
\begin{eqnarray}\label{va8}
v\in L^r(\mathbb{R}^3),\ \mbox{ for all } r\geq \frac{q}{p-1} \mbox{ such that } \frac{1}{r}\geq \frac{p-1}{q}-\frac{2}{3}.
\end{eqnarray}
Next we follow the arguments of Cazenave \cite{TC} to increase the index $r$.

For $j\geq 0$, we define $r_j$ as:
$$\frac{1}{r_j}=-\delta (p-1)^j+\frac{2}{3(p-2)},\ \mbox{ with } \delta= \frac{2}{3(p-2)}-\frac{1}{p}.$$
Since $p\in [3,6)$, then $\delta>0$ and $\frac{1}{r_j}$ is decreasing with $\frac{1}{r_j} \to -\infty$ as $j \to \infty$. Thus there exists some $k>0$ such that
$$\frac{1}{r_i}>0\ \mbox{ for }\ 0\leq i \leq k;\ \frac{1}{r_{k+1}}\leq 0.$$
Now we claim that $v\in L^{r_k}(\mathbb{R}^3)$. Indeed, $r_0=p$ as $j=0$ and it is trivial that $v\in L^{r_0}(\mathbb{R}^3)$. If we assume that 
$v\in L^{r_{i}}(\mathbb{R}^3)$ for $0 \leq i < k$, then by (\ref{va6}) and (\ref{keys}), we have $u\in L^{r_{i}}(\mathbb{R}^3)$. Thus following (\ref{va8}) we obtain
$$v\in L^r(\mathbb{R}^3),\ \mbox{ for all } r\geq \frac{r_{i}}{p-1} \mbox{ such that } \frac{1}{r}\geq \frac{p-1}{r_{i}}-\frac{2}{3}=\frac{1}{r_{i+1}}.$$
In particular, $v\in L^{r_{i+1}}(\mathbb{R}^3)$ and we conclude this claim by induction. Now
since $v\in L^{r_k}(\mathbb{R}^3)$ it follows from (\ref{va6}) and (\ref{keys})  that $u\in L^{r_k}(\mathbb{R}^3)$ and we get that
$$v\in L^r(\mathbb{R}^3),\ \mbox{ for all } r\geq \frac{r_k}{p-1} \mbox{ such that } \frac{1}{r}\geq \frac{p-1}{r_k}-\frac{2}{3}=\frac{1}{r_{k+1}}.$$
Since $1/r_{k+1}<0$ we obtain that $v\in \bigcap_{3\leq \alpha \leq \infty}L^{\alpha}(\mathbb{R}^3)$ and thus also $u\in \bigcap_{3\leq \alpha \leq +\infty}L^{\alpha}(\mathbb{R}^3)$.\\

At this point we have shown that
$$H u= u-\phi_{u} u + |u|^{p-2}u$$
with for all $\alpha \in [3,\infty] $,
$$ u\in L^{\alpha}\cap L^{\infty}(\mathbb{R}^3), \quad \phi_{u} u\in L^{\frac{6\alpha}{6+\alpha}}(\mathbb{R}^3) \quad \mbox{and} \quad  |u|^{p-2}u \in L^{\frac{\alpha}{p-1}}\cap L^{\infty} (\mathbb{R}^3).$$  Since $ \frac{6\alpha}{6+\alpha} \in [3,6] $ for $\alpha \in [6,\infty]$, by interpolation and (\ref{va1}) we obtain that
\begin{eqnarray}\label{va9}
u \in W^{2,\frac{6\alpha}{6+\alpha}}(\mathbb{R}^3)\ \mbox{ for any }\ \alpha \in [6,+\infty].
\end{eqnarray}
Thus by Sobolev's embedding, $u \in L^{\infty}(\mathbb{R}^3)\bigcap C^1(\mathbb{R}^3)$. Also there exists a sequence $\{u_n\}\subset C_c^1 (\mathbb{R}^3)$ such that $u_n \rightarrow u$ in $W^{2,\frac{6\alpha}{6+\alpha}}(\mathbb{R}^3)$. When $\alpha>6$, $W^{2,\frac{6\alpha}{6+\alpha}}(\mathbb{R}^3)\hookrightarrow  L^{\infty}(\mathbb{R}^3)$. Thus $u_n \rightarrow u$ uniformly in $\mathbb{R}^3$ and we conclude that $u(x)\rightarrow 0$ as $|x| \to \infty $.\\

\emph{Step 2:} Exponential decay estimate.\\

First we show that $\phi_{u} \in C^{0, \gamma}(\mathbb{R}^3)$, $\forall \gamma \in (0,1)$ and that there exists a constant $C_0>0$ such that
\begin{eqnarray}\label{va10}
\phi_{u} \geq \frac{C_0}{|x|},\ \ \mbox{ for all  }\ |x| \geq 1.
\end{eqnarray}
Since $\phi_{u} \in \mathcal{D}^{1,2}(\mathbb{R}^3)$ solves the equation $-\Delta \Phi=4 \pi |u|^2$ and $u \in L^6(\R^3)$ by elliptic regularity  $\phi_{u}\in W_{loc}^{2,3}(\mathbb{R}^3)$,
% because for any $R>0$,
%\begin{eqnarray}\label{dy1}
%\|\phi_{u}\|_{W^{2,3}(B_R)} &\leq& C \left ( \|\Delta \phi_{u}\|_{L^3(2B_R)} +\|\phi_{u}\|_{L^3(2B_R)}  \right)   \nonumber \\
%&\leq& \widetilde{C} \left ( \|u\|_{L^6(2B_R)} +\|\phi_{u}\|_{L^6(2B_R)}  \right),
%\end{eqnarray}
Thus by Sobolev's embedding, $\phi_{u} \in C^{0, \gamma}(\mathbb{R}^3)$, $\forall \gamma \in (0,1)$. In particular
$$C_0= \min_{\partial B_1}\phi_{u}(x) >0$$
where  $B_R=\{x\in \mathbb{R}^3: |x|\leq R \}$.
Indeed, if $\phi_{u}(x_0)=0$ at some point $x_0\in \mathbb{R}^3$ with $|x_0|=1$, then $u(x)=0$ a.e. in $\mathbb{R}^3$.

Now for an arbitrary $R_0>0$, let $w_1=\phi_{u}-\frac{C_0}{|x|}$. Then
$$\left\{\begin{matrix}
-\Delta w_1 =4 \pi u^2\geq 0, &\mbox{ in }& B_{R_0} \setminus B_1;\\
w_1 \geq 0, &\mbox{ on }& \partial B_1;\\
w_1 \geq -\frac{C_0}{R}, &\mbox{ on }& \partial B_{R_0}
\end{matrix}\right.$$
and the maximum principle yields that
$$ w_1 \geq -\frac{C_0}{R_0},\ \mbox{ in } B_{R_0} \setminus B_1. $$
Letting $R_0 \to \infty$, it follows that $w_1 \geq 0$ in $\mathbb{R}^3\setminus B_1$ and thus (\ref{va10}) holds.\\

Now we denote by $u^+ (u^-)$ the positive (negative) part of $u$, namely $u^+(x)=\max\{u(x),0\}$ and $u^-(x)=\max\{-u(x),0\}$.

By Kato's inequality, we know that $\Delta u^+\geq \chi[u\geq 0]\Delta u$, see \cite{BP}. Thus
\begin{eqnarray}\label{dy1.1}
-\Delta u^+ -\lambda u^+ + \phi_{u}\cdot u^+ \leq (u^+)^{p-1}\ \mbox{  in  }\ \mathbb{R}^3.
\end{eqnarray}
Let us show that there exist constants $\widetilde{C}>0$ and $R_1>0$ such that
\begin{eqnarray}\label{dy2}
u^+(x) \leq  \widetilde{C} \phi_{u}(x)\ \ \mbox{for } |x|>R_1.
\end{eqnarray}
To prove this, we consider $w_2=u^+ - \phi_{u}- \frac{d}{|x|},$ for a constant $d>0$. Then (\ref{dy1.1}) and $\lambda \leq 0$ imply that
$$-\Delta w_2\leq (u^+)^{p-1}-4 \pi u^2,\ \ \mbox{in } |x|\geq 1.$$

Since $\lim_{|x|\to \infty}u(x)\to 0$ and $p>3$, then
$(u^+)^{p-1}-4 \pi u^2 \leq 0 $ holds in $|x|\geq R_1$ for some $R_1 >0$ large enough. Thus for any $R \geq R_1$ and taking  $d>0$ large enough we have
\begin{eqnarray*}
\left\{\begin{matrix}
-\Delta w_2 \leq 0,  &\mbox{ in }&\  B_R\setminus B_{R_1}; \\
w_2 \leq 0, &\mbox{  on  }&\ \partial B_{R_1};\\
w_2 \leq \max_{\partial B_R} u^+-\frac{d}{R},\  &\mbox{  on  }&\  \partial B_{R}.
\end{matrix}\right.
\end{eqnarray*}
Then by the maximum principle, we have  $w_2 \leq \max_{\partial B_{R}}u^+-\frac{d}{R}$ in $B_R\setminus B_{R_1}$. Letting $R\to \infty$ we conclude that $w_2\leq 0$ in $\mathbb{R}^3 \setminus B_{R_1}$. This, together with (\ref{va10}), implies (\ref{dy2}).

>From (\ref{dy1.1}) we have for any $\sigma >0$ and since $\lambda \leq 0$,
\begin{eqnarray}\label{dy5}
-\Delta u^+ + \frac{\sigma}{|x|}u^+&\leq&  \frac{\sigma}{|x|}u^+ - \phi_{u}u^+ +\lambda u^+ + (u^+)^{p-1}  \nonumber \\
&\leq&\left (\frac{\sigma}{|x|}- \phi_{u} +(u^+)^{p-2} \right )u^+.
\end{eqnarray}
Using (\ref{va10}) and (\ref{dy2}), for $|x|\geq R_1>1$, by choosing $0<\sigma <C_0$, we have
\begin{eqnarray*}
\frac{\sigma}{|x|}- \phi_{u}  + (u^+)^{p-2} &\leq& \frac{\sigma}{C_0}\cdot \phi_{u}- \phi_{u} +(u^+)^{p-2} \\
&\leq& -(1-\frac{\sigma}{C_0})\widetilde{C}^{-1}u^+ + (u^+)^{p-2}\\
&=& \left(-(1-\frac{\sigma}{C_0})\widetilde{C}^{-1} + (u^+)^{p-3}  \right)\cdot u^+,
\end{eqnarray*}
where $(1-\frac{\sigma}{C_0})\widetilde{C}^{-1}>0$. Since $p\geq 3$ and $u(x)\to 0$ as $|x|\to \infty$, for $R_1>1$ sufficiently large, we obtain that $-(1-\frac{\sigma}{C_0})\widetilde{C}^{-1} + (u^+)^{p-3}\leq 0$ in $|x|\geq R_1$. Thus it follows from (\ref{dy5}) that
\begin{eqnarray}\label{dy6}
-\Delta u^+ + \frac{\sigma}{|x|}u^+ \leq 0,\quad \mbox{in } \mathbb{R}^3 \setminus B_{R_1}.
\end{eqnarray}
If we denote $\bar{C}_1=\max_{\partial B_{R_1}}u^+$, applying the maximum principle, we thus obtain
\begin{eqnarray}\label{dy7}
u^+ \leq \bar{C}_1\cdot \bar{w},\quad \mbox{in } \mathbb{R}^3 \setminus B_{R_1}
\end{eqnarray}
where $\bar{w}$ is the radial solution of
$$\left\{\begin{matrix}
-\Delta \bar{w}+\frac{\sigma}{|x|}=0,\ &\mbox{ if }& |x|>R_1;   \\
\bar{w}(x)=1,\ &\mbox{ if }& |x|=R_1;  \\
\bar{w}(x)\to 0,\ &\mbox{ if }& |x|\to \infty.
\end{matrix}\right.$$
Now $\bar{w}$  satisfies (cfr. \cite{AMR} Section 4),
\begin{eqnarray}\label{dy8}
\bar{w}(x) \leq \frac{C}{|x|^{3/4}}e^{-2C'\sqrt{|x|}},\ \ \forall\ |x|>R'
\end{eqnarray}
for some $C>0, C'>0$ and $R'>0$.\\

Finally we observe that if $u$ is a solution of (\ref{eq}), then $-u$ is also a solution. Thus since $u^-=(-u)^+$,  following the same arguments, we obtain that there exists a constant $\bar{C}_2>0$ such that
\begin{eqnarray}\label{dy9}
u^- \leq \bar{C}_2\cdot \bar{w},\quad \mbox{in } \mathbb{R}^3 \setminus B_{R_1}.
\end{eqnarray}

Hence $|u|=u^++u^-\leq (\bar{C}_1+\bar{C}_2)\bar{w},\ \mbox{ in } \mathbb{R}^3 \setminus B_{R_1}$ for $R_1>0$ sufficiently large.
At this point we see from (\ref{dy8})  that $u \in E$ satisfies the exponential decay (\ref{decay}). In particular  $u\in L^2(\R^3)$ and then also $u \in H^1(\R^3)$.
\end{proof}

\begin{lem}\label{constant}
There exists $c_{\infty}>0$ such that for all $c \geq c_{\infty}$ the function $c \to \gamma(c)$ is constant. Also if for a $c \geq c_{\infty}$ there exists a couple $(u_c, \lambda_c) \in H^1(\R^3) \times \R$ solution of \eqref{eq} with $||u_c||_2^2 = c$ and $F(u_c) = \gamma(c)$ then necessarily $\lambda_c =0.$
\end{lem}
\begin{proof}
>From Lemma \ref{description} and \cite{IARU} we know that there exists grounds states of the free functional $F(u)$ which are real. From  Theorem \ref{th3.1} we know that any ground state belongs to $H^1(\R^3)$. Let $u_0 \in H^1(\R^3)$ be one of these ground states and set $c_0 = ||u_0||^2_2$. Then, by Lemma \ref{pohoz}, $u_0 \in V(c_0)$ and using Lemma \ref{boundbelow} we get
$$F(u_0) \geq \gamma(c_0) \geq m = F(u_0).$$
Thus necessarily $\gamma(c_0) = m$. Now since $c \to \gamma(c)$ is non increasing, still by Lemma \ref{boundbelow}, we deduce that $\gamma(c) = \gamma(c_0)$ for all $c \geq c_0$. Now let $(u_c, \lambda_c) \in H^1(\R^3) \times \R$ be a solution of (\ref{eq}) with $||u_c||_2^2 = c$ and $F(u_c) = \gamma(c)$. It is not possible to have $\lambda_c >0$ otherwise, by Lemma \ref{strictdecrease}, the function $c\to \gamma(c)$ would be strictly increasing around $c>0$ contradicting Lemma \ref{prop1}.
Also $\lambda_c <0$ is not possible since by Lemma \ref{strictdecrease} it would imply that $c \to \gamma(c)$ is strictly decreasing around $c>0$ in  contradiction with the fact that it is constant. Then necessarily $\lambda_c=0$.
\end{proof}
\begin{remark}\label{versnonexits}
We see, from Theorem \ref{naturalconstraint} and Lemma \ref{constant}, that if $\gamma(c)$ is reached, say by a $u_c \in H^1(\R^3)$ with $c >0$ large enough, then $u_c$ is a ground state of $F(u)$ defined on $E$. It is unlikely that ground states exist for an infinity of value of $c>0$. So we conjecture that there exists a $c_{lim} >0$ such that for  $c \geq c_{lim}$ there are no critical points for $F(u)$ constrained to $S(c)$ at the ground state level $\gamma(c)$.
\end{remark}

\begin{proof}[Proof of Theorem \ref{maingamma}]
Obviously, points (i), (ii), (iv), (v) of Theorem \ref{maingamma} follow directly from Lemmas \ref{prop1}, \ref{continuity}, \ref{limzero}, \ref{liminf}, \ref{constant} and Lemmas \ref{pohoz}, \ref{strictdecrease} conclude point (iii).
\end{proof}

%If we assume that we work in the subspace of radially symmetric functions $H_r^1(\R^3)$ instead of all $H^1(\R^3)$ we can derive additional information on $c \to \gamma(c)$. We point out however that we still do not know if the ground states of (\ref{eq}) in $H_r^1(\R^3)$ and $H^1(\R^3)$ coincide. See Remark 3.9 of \cite{IARU} and \cite{GEPRVI} in that direction.

\section{Global existence and strong instability} \label{Section8}

We introduce the following result about the locally well-posedness of the Cauchy problem to the equation \eqref{main} (see Cazenave \cite{TC}, \emph{Theorem 4.4.6} and \emph{Propostion 6.5.1} or Kikuchi's Doctoral thesis \cite{HK}, \emph{Chapter 3}).\\

\begin{prop}\label{prop2.1}
Let $p \in (2,6)$, for any $u_0 \in H^1(\mathbb{R}^{3},\mathbb{C})$, there exists $T=T(\left \| u_0\right \|_{H^1})>0$ and a unique solution $u(t) \in C([0,T), H^1(\mathbb{R}^{3},\mathbb{C}))$ of the equation \eqref{main} with initial datum $u(0)=u_0$ satisfying
\begin{eqnarray*}
F(u(t))=F(u_0), \,  \left \| u(t) \right \|_2=\left \| u_0 \right \|_2 \mbox{ for any }   t \in [0,T).
\end{eqnarray*}
In addition, if $u_0 \in H^1(\mathbb{R}^{3},\mathbb{C})$ satisfies $| x| u_0 \in L^2(\mathbb{R}^{3},\mathbb{C})$, then the virial identity
$$\frac{\mathrm{d}^2}{\mathrm{d} t^2} \left \|xu(t)\right\|_2^2=8Q(u),$$
holds for any $ t\in [0,T)$.
\end{prop}

\begin{proof}[Proof of Theorem \ref{global}]
Let $u(x,t)$ be the solution of \eqref{main} with $u(x,0)=u_0$ and $T_{max} \in (0, \infty]$ its maximal time of existence. Then classically we have either
$$T_{max}=+\infty$$
or
\begin{equation}\label{blowup}
T_{max} < + \infty \quad \mbox{ and } \lim_{t \rightarrow T_{max}}||\nabla u(x,t)||_2^2=\infty.
\end{equation}
Since
$$
F(u(x,t))-\frac{2}{3(p-2)}Q(u(x,t))=\frac{3p-10}{6(p-2)}A(u(x,t))+\frac{3p-8}{12(p-2)}B(u(x,t))
$$
and  $F(u(x,t))=F(u_0)$ for all $t<T_{max}$, if \eqref{blowup} happens then,
we get
$$\lim_{t \rightarrow T_{max}}Q(u(x,t))=-\infty.$$
By continuity it exists $t_0 \in (0, T_{max})$ such that $Q(u(x, t_0))=0$ with
$F(u(x,t_0))=F(u_0)<\gamma(c)$. This contradicts the definition $\gamma(c) = \inf_{u\in V(c)}F(u)$.
\end{proof}

\begin{remark}\label{rem1.3'}
For $p \in (\frac{10}{3},6)$ and any $c>0$ the set $\mathcal{O}$ is not empty. Indeed for an arbitrary  but fixed  $u\in S(c)$ $u^{t}(x)=t^{\frac{3}{2}}u(t x)$. Then $u^{t} \in S(c)$ for all $ t>0$ and
$$Q(u^{t})= t^2A(u)+\frac{t}{4}B(u)-\frac{3(p-2)}{2p}t^{\frac{3(p-2)}{2}}C(u),$$
$$F(u^{t})= \frac{t^2}{2}A(u)+\frac{t}{4}B(u)-\frac{t^{\frac{3(p-2)}{2}}}{p}C(u).$$
We observe that $ F(u^t) \to 0$ as $t \to 0$. Also, since $\frac{3(p-2)}{2} >1$, we have $Q(u^t) >0$ when $t>0$ is sufficiently small. This proves that $\mathcal{O}$ is not empty.
\end{remark}

%%%%%%%%%%%%%%%%%%%%%%%%%%%%%%%%%%%%%%%%%%%%%%%%%%%%%%%%%%%%%%%%%%%%%%%%%%%%%%%%%%%%%%%%%%%%%%%%%%%%%%%%%%%%%%%%%%%%%%%%%%%

%%%%%%%%%%%%%%%%%%%%%%%%%%%%%%%%%%%%%%%%%%%%%%%%%%%%%%%%%%%%%%%%%%%%%%%%%%%%%%%%%%%%%%%%%%%%%%%%%%%%%%%%%%%%%%%%%%%%%%%%%%%

\begin{proof}[Proof of the Theorem \ref{th2.1}] For any $c>0$, let $u_c \in \mathcal{M}_c $
and define the set
\begin{eqnarray*}
\Theta= \left \{v \in H^1(\mathbb{R}^3)\setminus \{0\}:\  F(v)<F(u_c),\ \left \| v \right \|_2^2=\left \| u_c \right \|_2^2, \ Q(v)<0   \right \}.
\end{eqnarray*}
The set $\Theta$ contains elements arbitrary close to $u_c$ in $H^1(\R^3)$. Indeed, letting $v_0(x)=u_c^{\lambda}=\lambda^{\frac{3}{2}}u_c(\lambda x)$, with $\lambda <1$, we see from Lemma \ref{growth} that $v_0 \in \Theta$ and that $v_0 \to u_c$ in $H^1(\mathbb{R}^3)$ as $\lambda \to 1$. \medskip

Let $v(t)$ be the maximal solution of \eqref{main} with initial datum $v(0)=v_0$ and  $T \in (0, \infty]$ the maximal time of existence. Let us show that $v(t) \in \Theta$ for all $ t \in [0,T)$. From the conservation laws
$$\left \| v(t) \right \|_2^2 = \left \| v_0 \right \|_2^2 =\left \| u_c \right \|_2^2,$$ and  $$F(v(t))=F(v_0)<F(u_c).$$ Thus it is enough to verify $Q(v(t))<0$. But $Q(v(t)) \neq 0$  for any $t \in (0, T)$. Otherwise, by the definition of $\gamma(c)$, we would get for a $t_0 \in (0,T)$ that $F(v(t_0)) \geq F(u_c)$ in contradiction with $F(v(t)) < F(u_c)$. Now by continuity of $Q$ we get that $Q(v(t))<0$ and thus that $v(t) \in \Theta$ for all $ t \in [0,T)$.
Now we claim that there exists $\delta >0$, such that
\begin{eqnarray}\label{le2.4}
Q(v(t))\leq - \delta , \ \forall t \in [0,T).
\end{eqnarray}
Let $t \in [0,T)$ be arbitrary but fixed and set $v = v(t)$. Since $Q(v) <0$ we know by Lemma \ref{growth} that $\lambda^{\star}(v)<1$ and that $\lambda \longmapsto F(v^{\lambda})$ is concave on $[\lambda ^{\star}, 1)$.
Hence
\begin{eqnarray*}
F(v^{\lambda^{\star}})-F(v) &\leq& (\lambda^{\star}-1) \frac{\partial}{\partial \lambda}F(v^{\lambda})\mid _{\lambda =1} \\
&=& (\lambda^{\star}-1) Q(v).
\end{eqnarray*}
Thus, since $Q(v(t))<0$, we have
$$F(v)-F(v^{\lambda^{\star}})\geq (1- \lambda^{\star})Q(v)\geq Q(v).$$
It follows from $F(v)=F(v_0)$ and $v^{\lambda^{\star}} \in V(c)$ that
$$Q(v) \leq F(v)-F(v^{\lambda^{\star}})\leq F(v_0)-F(u_c).$$
Then letting  $\delta=F(u_0)-F(v_0)>0$  the claim is established. To conclude the proof of the theorem we use Proposition \ref{prop2.1}. Since $v_0(x)=u_c^{\lambda}$ we have that
$$\int_{\mathbb{R}^3}| x|^2 |v_0|^2dx=\int_{\mathbb{R}^3}| x|^2 |u_c^{\lambda}|^2dx=\lambda^2 \int_{\mathbb{R}^3}| y|^2 |u_c(y)|^2dy.$$
Thus, from Lemma \ref{minimizer} and Theorem \ref{th3.1}, we obtain that
\begin{eqnarray}\label{2.3}
\int_{\mathbb{R}^3}| x|^2 |v_0|^2dx<\infty.
\end{eqnarray}
Applying Proposition \ref{prop2.1} it follows that
$$\frac{\mathrm{d}^2}{\mathrm{d} t^2} \left \|xv(t)\right\|_2^2=8Q(v).$$
Now by (\ref{le2.4}) we deduce that $v(t)$ must blow-up in finite time, namely that (\ref{blowup}) hold. Recording that $v_0$ has been taken arbitrarily close to $u_c$, this ends the proof of the theorem.
\end{proof}

\begin{proof}[Proof of Theorem \ref{cor1.1}] For $p\in (\frac{10}{3}, 6)$,  let $u_0$ be a ground state of equation (\ref{eqfree}). From Theorem \ref{th3.1} we know that $u_0\in H^{1}(\R^3)$, thus we can set
$$c_0:=\|u_0\|_2^2.$$
>From Lemma \ref{pohoz}, we have $Q(u_0)=0$. Thus $u_0\in V(c_0)$ and it follows from (\ref{mini}) and Lemma \ref{boundbelow} that
$$F(u_0)\geq \gamma(c_0)\geq m=F(u_0).$$
Hence $F(u_0)=\inf_{u\in V(c_0)}F(u)$, which means that $u_0$ minimizes $F(u)$ on $V(c_0)$. Thus applying Theorem \ref{th2.1}, we end the proof.
\end{proof}

\section{Comparison with the nonlinear Schr\"odinger case} \label{Section7}

In \cite{LJ} the existence of critical points of
\begin{eqnarray}
\tilde F(u)=\frac{1}{2}\left \|\triangledown  u \right \|_2^2-\frac{1}{p}\| u \|_p^p,\quad \ u \in H^1(\mathbb{R}^N).
\end{eqnarray}
constrained to $S(c)$ was considered under the condition:
$$ (C) : \frac{2N+4}{N}< p <\frac{2N}{N-2}, \mbox{ if }  N\geq 3 \mbox{ and } \frac{2N+4}{N}< p \mbox{ if } N=1,2.$$
In our notation it is proved in \cite{LJ} that $\tilde F(u)$ has a mountain pass geometry on $S(c)$ in the sense that
$$\tilde \gamma(c)=\inf_{g \in \Gamma_c } \max_{t \in [0,1]}\tilde F(g(t)) > \max\{\tilde F(g(0)), \tilde F (g(1))\}$$
where
$$\tilde \Gamma_c =\{g \in C([0,1],S(c)), \ g(0) \in A_{K_c}, \tilde F (g(1))<0\},$$
and $A_{K_c}= \{u \in S(c):\ \left \| \triangledown u \right \|_2^2 \leq K_c\}$. Also we have
\begin{lem}(\cite{LJ} Theorem 2)\label{lemma3}
For $N\geq 1$ and any $c>0$, under the condition $(C)$, the functional $\tilde F(u)$ admits a critical point $u_c$ at the level $\tilde \gamma(c)$ with $\|u_c\|_2^2=c$ and there exists $\lambda_c<0$ such that $(\lambda_c, u_c)$ solves weakly the following Euler-Lagrange equation associated to the functional $\tilde F(u)$ :
\begin{eqnarray}\label{NLS}
-\Delta u - \lambda u= |u|^{p-2}u.
\end{eqnarray}
\end{lem}
\begin{lem}(\cite{LJ} Corollary 3.1 and Theorem 3.2)\label{lemma4}
For $N\geq1$, as
$c \to 0$
$$\left \{ \begin{matrix}
\|\nabla u_c\|_2^2 \rightarrow \infty, \\
\lambda_c \rightarrow -\infty .
\end{matrix}\right.$$
and as $c\to +\infty$,
$$\left \{ \begin{matrix}
\|\nabla u_c\|_2^2 \rightarrow 0, \\
\lambda_c \rightarrow 0 .
\end{matrix}\right.$$
\end{lem}
Using the above two results we now prove
\begin{lem}\label{prop2}
For $N\geq 1$, under the condition $(C)$, the function $c\longmapsto \tilde \gamma(c)$ is strictly decreasing. In addition, we have
\begin{eqnarray}\label{limites}
\left \{ \begin{matrix}
\tilde \gamma(c) \rightarrow +\infty,  &as& c \to 0,\\
\tilde \gamma(c) \rightarrow 0 ,    &as& c \to \infty.
\end{matrix}\right.
\end{eqnarray}
\end{lem}
\begin{proof}
Arguing as in the proof of Lemma \ref{lm1} we can deduce that
\begin{eqnarray}\label{17}
\tilde \gamma(c)= \inf_{u\in S(c)} \max_{t>0}\tilde F(u^t) =\inf_{u\in \tilde V(c)}\tilde F(u).
\end{eqnarray}
Here $\tilde V(c) = \{u \in H^1(\R^N): \tilde Q(u) =0 \}$ with
$$\tilde Q(u) =\left \|\triangledown  u \right \|_2^2-\frac{N(p-2)}{2p}\| u \|_p^p$$
 and $u^t(x)=t^{\frac{N}{2}}u(tx)$ for $t>0$. To show that $c \to \tilde \gamma(c)$ is strictly decreasing we just need to prove that: for any $c_1<c_2$, there holds $\tilde \gamma(c_2)< \tilde\gamma(c_1)$. By \eqref{17} we have
 $$ \tilde \gamma(c_1)=\inf_{u\in S(c_1)} \max_{t>0}\tilde F(u^t) \quad \mbox{and} \quad \tilde \gamma(c_2)=\inf_{u\in S(c_2)} \max_{t>0}\tilde F(u^t)$$
where
$$\tilde F(u^t)=\frac{t^2}{2}\left \|\triangledown  u \right \|_2^2-\frac{t^{\frac{N}{2}(p-2)}}{p}\| u \|_p^p.$$
After a simple calculation, we get
\begin{eqnarray}
\max_{t>0}\tilde F(u^t)=\widetilde{c}(p)\cdot \left (\frac{1}{2}\|\nabla u\|_2^2  \right )^{\frac{N(p-2)}{N(p-2)-4}}\cdot \left (\frac{1}{p}\|\nabla u\|_p^p  \right )^{-\frac{4}{N(p-2)-4}}
\end{eqnarray}
with
$$\widetilde{c}(p)=\left (\frac{4}{N(p-2)}\right )^{\frac{4}{N(p-2)-4}}\cdot \frac{N(p-2)-4}{N(p-2)}>0.$$
By Lemma \ref{lemma3}, we know that $\gamma(c_1)$ is attained, namely that there exists $u_1\in S(c_1)$, such that $\tilde \gamma(c_1)= \tilde F(u_1)=\max_{t>0}\tilde F(u_1^t)$. Then using the scaling $u_{\theta}(x)=\theta^{1-\frac{N}{2}}u_1(\frac{x}{\theta})$, we have
$$\|u_{\theta}\|_2^2=\theta^2\|u_1\|_2^2, \quad \|\nabla u_{\theta}\|_2^2=\|\nabla u_1\|_2^2 \quad \mbox{and} \quad \|u_{\theta}\|_p^p=\theta^{(1-\frac{N}{2})p+N}\|u_1\|_p^p.$$
Thus we can choose $\theta>1$ such that $u_{\theta}\in S(c_2)$. Under the condition $(C)$, we have $(1-\frac{N}{2})p+N>0$ for $N\geq 1$ and thus $\|u_{\theta}\|_p^p > \| u_1\|_p^p$. Now we have
\begin{eqnarray*}
\max_{t>0}\tilde F(u_{\theta}^t) &=& \widetilde{c}(p)\cdot \left (\frac{1}{2}\|\nabla u_{\theta}\|_2^2  \right )^{\frac{N(p-2)}{N(p-2)-4}}\cdot \left (\frac{1}{p}\|u_{\theta}\|_p^p  \right )^{-\frac{4}{N(p-2)-4}} \\
&<&\widetilde{c}(p)\cdot \left (\frac{1}{2}\|\nabla u_1\|_2^2  \right )^{\frac{N(p-2)}{N(p-2)-4}}\cdot \left (\frac{1}{p}\| u_1\|_p^p  \right )^{-\frac{4}{N(p-2)-4}} \\
&=&\max_{t>0}\tilde F(u_1^t),
\end{eqnarray*}
which implies that
\begin{eqnarray}
\tilde \gamma(c_1)=\max_{t>0}\tilde F(u_1^t)>\max_{t>0}\tilde F(u_{\theta}^t)\geq \tilde \gamma(c_2).
\end{eqnarray}
Finally, from Lemma 2.7 of \cite{LJ} we know that, for any $c >0$, $\tilde Q(u_c)=0$. Thus we can write
$$
\tilde \gamma(c)= \frac{N(p-2)-4}{2N(p-2)}\left \|\triangledown  u_c \right \|_2^2$$
and \eqref{limites} directly follows from Lemma \ref{lemma4}.
\end{proof}

Finally in analogy with Theorems   \ref{naturalconstraint} and  \ref{th2.1} we have
\begin{remark}
Let 
\begin{eqnarray}\label{minimizersetNLS}
\mathcal{\tilde{M}}_c := \{u_c\in \tilde{V}(c)\ : \ \tilde{F}(u_c)=\inf_{u\in \tilde{V}(c)}\tilde{F}(u)\}.
\end{eqnarray}
Then for any $u_c \in \mathcal{\tilde{M}}_c$ there exists a $\lambda_c <0$ such that $(u_c, \lambda_c) \in H^1(\R^N) \times \R$ solves 
(\ref{NLS}) and the standing wave solution $e^{-i \lambda_c t}u_c$ of (\ref{evolution2}) is strongly unstable. 

The proof of these statements is actually simpler than the ones for (\ref{main}) and thus we just indicate the main lines. We proceed as in Lemma \ref{minimizer} to show that for any $u_c \in \mathcal{\tilde{M}}_c$ there exists a $\lambda_c \in \R$ such that $(u_c, \lambda_c) \in H^1(\R^N) \times \R$ solves (\ref{NLS}). Indeed a version of Lemma \ref{growth} (and thus of Lemma \ref{mpestimate}) holds when $\tilde{F}(u)$ replaces $F(u)$ and this is precisely Lemma 8.2.5 in \cite{TC}. Now if for a $\lambda \in \R$, $u \in S(c)$ solves
\begin{equation}\label{91}
- \Delta u - |u|^{p-2}u = \lambda u,
\end{equation}
on one hand, multiplying (\ref{91}) by $u \in S(c)$ and integrating we obtain 
\begin{equation}\label{92}
||\nabla u||_2^2 - ||u||_p^p = \lambda c.
\end{equation}
On the other hand, since solutions of (\ref{91}) satisfy $\tilde{Q}(u)=0$, we have
\begin{equation}\label{93}
||\nabla u||_2^2 - \frac{N(p-2)}{2p}||u||_p^p = 0.
\end{equation}
Thus, since under $(C)$ $N(p-2)/2p < 1$, we deduce that necessarily $\lambda <0$. To conclude the proof we just have to show that the standing wave $e^{-i \lambda_c t}u_c$ is strongly unstable. This can be done following the same lines as in the proof of Theorem \ref{th2.1}. Here the fact that $\lambda_c <0$ insures the exponential decay at infinity of $u_c \in S(c)$ which permits to use the virial identity in the blow-up argument (see also \cite{BECA}).
\end{remark}

%%%%%%%%%%%%%%%%%%%%%%%%%%%%%%%%%%%%%%%%%%%%%%%%%%%%%%%%%%%%%%%%%%%%%%%%%%%%%%%%%%%%%%%%%%%%%%%%%%%%%%%%%%%%%%%%%%%%%%%%

\end{document}